\documentclass[letterpaper,11pt]{article}
\usepackage[utf8x]{inputenc}
\usepackage[title]{appendix}
\usepackage{tikz}
\pgfdeclarelayer{bg}  
\pgfsetlayers{bg,main}
\usepackage[noTeX]{mmap}
\usepackage[left=1in, right=1in, bottom = 1in, top = 1in]{geometry}
\usepackage{hyperref} 
\hypersetup{colorlinks}
\usepackage{amsmath,amsfonts, amsthm, amssymb}
\usepackage{color}
\usepackage{enumerate}
\usepackage{hypbmsec}
\usepackage{bbm}
\usepackage{upgreek}
\usepackage{dsfont}
\usepackage{mathtools}
\usepackage{graphicx}
\usepackage{caption}
\usepackage{subcaption}
\usepackage[section]{placeins}
\usepackage{tikz}
\usetikzlibrary{calc}
\usetikzlibrary{decorations.pathreplacing}
\usetikzlibrary{arrows}
\usepackage{pgffor}
\usepackage{longtable}
\usepackage{comment}
\usepackage{multirow}
\usepackage{rotating}
\usepackage{bm}
\usepackage{mdframed}
\usepackage{todonotes}

\usepackage{tocloft}

\usepackage{mathtools,tabstackengine}
\TABstackMath% STACK MATRIX IN MATH MODE
\setstacktabbedgap{10pt}% INTER-COLUMN GAP SIZE
\setstackgap{L}{1.5\baselineskip}% INTER-ROW BASELINESKIP

\theoremstyle{plain}
\newtheorem{theorem}{Theorem}[section]
\newtheorem{proposition}[theorem]{Proposition}
\newtheorem{corollary}[theorem]{Corollary}
\newtheorem{lemma}[theorem]{Lemma}

\theoremstyle{definition}
\newtheorem{remark}[theorem]{Remark}

\DeclarePairedDelimiter\floor{\lfloor}{\rfloor}

\newcommand{\Pb}{\mathbb{P}}

\newcommand{\Zphi}{\mathcal{Z}^{\Phi}}
\newcommand{\Zt}{\mathcal{Z}^{\mathsf{t}}}
\newcommand{\Zj}{\mathcal{Z}^{j}}
\newcommand{\Xphi}{\mathcal{X}^{\Phi}}
\newcommand{\Xt}{\mathcal{X}^{\mathsf{t}}}
\newcommand{\Xj}{\mathcal{X}^{j}}

\newcommand{\Yt}{\mathcal{Y}^{\mathsf{t}}}
\newcommand{\Yj}{\mathcal{Y}^{j}}
\newcommand{\Fphi}{\mathcal{F}^{\Phi}}
\newcommand{\Ft}{\mathcal{F}^{\mathsf{t}}}
\newcommand{\Fj}{\mathcal{F}^{j}}
\newcommand{\Fi}{\mathcal{F}^{\mathsf{inv}}}
\newcommand{\Gphi}{\mathcal{G}^{\Phi}}
\newcommand{\Gt}{\mathcal{G}^{\mathsf{t}}}
\newcommand{\Gj}{\mathcal{G}^{j}}
\newcommand{\defeq}{\vcentcolon=}
\newcommand{\eqdef}{=\vcentcolon}

\newcommand{\AS}{\xrightarrow{\scriptscriptstyle{a.s.}}}

\newcommand{\Probto}{\xrightarrow{\scriptscriptstyle{\Prob}}}
\newcommand{\ProbSto}{\xrightarrow{\scriptscriptstyle{\ProbS}}}
\newcommand{\distto}{\xrightarrow{\scriptscriptstyle{\mathrm{d}}}}
\newcommand{\distSto}{\xrightarrow{\scriptscriptstyle{\mathrm{d},\S}}}
\newcommand{\stabto}{\xrightarrow{\scriptscriptstyle{\mathrm{st}}}}
\newcommand{\stabSto}{\xrightarrow{\scriptscriptstyle{\mathrm{st},\S}}}

\newcommand{\ind}[1]{\mathds{1}_{\{#1\}}}

\newcommand{\N}{\mathds{N}}
\newcommand{\Z}{\mathds{Z}}

\newcommand{\R}{\mathds{R}}

\newcommand{\C}{\mathds{C}}
\newcommand{\T}{\mathds{T}}

\newcommand{\U}{\mathcal{U}_{\infty}}

\newcommand{\Surv}{\mathcal{S}}

\newcommand{\Prob}{\mathds{P}}
\newcommand{\ProbS}{\mathds{P}^\S}
\DeclareMathOperator{\Var}{\mathrm{Var}}
\DeclareMathOperator{\Cov}{\mathrm{Cov}}

\newcommand{\E}{\mathds{E}}
\newcommand{\A}{\mathcal{A}}
\renewcommand{\L}{\mathcal{L}}
\newcommand{\F}{\mathcal{F}}

\newcommand{\cN}{\mathcal N}

\renewcommand{\S}{\mathcal{S}}

\newcommand{\eqdist}{%
	\mathrel{\vbox{\offinterlineskip\ialign{%
				\hfil##\hfil\cr
				$\scriptscriptstyle\mathrm{law}$\cr
				\noalign{\kern.2ex}
				$=$\cr
}}}}

\newcommand{\type}{\mathrm t}
\newcommand{\kk}[1]{\textcolor{purple}{#1}}

\title{Gaussian fluctuations for the two urn model}

\author{Konrad Kolesko and  Ecaterina Sava-Huss}

\date{\today}
\begin{document}
\maketitle

\begin{abstract}
We introduce a modification of the generalized P\'olya urn model containing two urns and we study the number of balls $B_j(n)$ of a given color $j\in\{1,\ldots,J\}$, $J\in\N$ added to the urns after $n$ draws. We provide sufficient conditions under which the random variables $(B_j(n))_{n\in\N}$  properly normalized and centered converge weakly to a limiting random variable. The result reveals a similar trichotomy as in the classical case with one urn, one of the main differences being that in the scaling we encounter 1-periodic continuous functions. Another difference in our results compared to the classical urn models is that the phase transition of the second order behavior occurs at $\sqrt{\rho}$ and not at $\rho/2$, where $\rho$ is the dominant eigenvalue of the mean replacement matrix.

\end{abstract}


%\begin{small}
%\tableofcontents
%\end{small}
\section{Introduction}

\textbf{A brief overview on classical P\'olya urn models.}
For $J\in\N$, a $J$-dimensional P\'olya urn process $(B(n))_{n\in\N}$ is a $\N^J$-valued stochastic process which represents the 
evolution of an urn containing balls of $J$ different colors denoted by $1,2,\ldots, J$. The initial composition
of the urn can be specified by a $J$-dimensional vector $B(0)=\big(B_1(0),\ldots,B_J(0)\big)$, the $j$-th coordinate $B_j(0)$ of $B(0)$ representing the number of balls of color $j$ present in the urn at the beginning of the process, i.e.~at time $0$. At each subsequent timestep $n\geq 1$, we pick a ball uniformly at random 
, inspect its color, put it back  into the urn with an additional, random collection of $L^{(j)}=(L^{(j,1)},\ldots,L^{(j,J)})$  balls, if the picked ball has color $j$ (which happens with probability proportional to the already existing number of balls of color $j$ in the urn). %\textcolor{blue}{ If the model is without replacement, than one does not return to the urn the picked ball, like in the model we consider throughout this work.} 
This way of adding balls may be summed up by the so-called replacement matrix $L$, which in our case is a random matrix $L$ defined as following. For $J\in\N$, we write $[J]:=\{1,\ldots,J\}$, and we consider  a sequence  $(L^{(j)})_{j\in [J]}$ of $J$ independent $\N^J$-valued random (column) vectors. We denote by $L$ the $J\times J$ random matrix with column vectors $L^{(j)}$, that is $L=\left(L^{(1)},L^{(2)},\ldots,L^{(J)}\right)$, and 
%that is:
%$$L=\left(L^{(1)},L^{(2)},\ldots,L^{(J)}\right)
%=\begin{pmatrix}
%	L^{(1,1)} & L^{(2,1)} & \cdots & L^{(J,1)} \\
%	L^{(1,2)} & L^{(2,2)} & \cdots & L^{(J,2)} \\
%	\vdots    & \vdots    & \vdots & \vdots    \\
%	L^{(1,J)} & L^{(2,J)} & \cdots & L^{(J,J)}
%\end{pmatrix}
%$$  
by $a_{ij}=\E\left[L^{(j,i)}\right]$ the expectation of $L^{(j,i)}$, for all $i,j\in [J]$, and by $A=(a_{ij})_{i,j\in[J]}$ so  $\E L=A$.
% In terms of the matrix $L$, in the classical urn model,  any picked ball of color $j\in[J]$, is replaced by a collection of balls in which for each $i\in[J]$ there are  $L^{(j,i)}$  balls of color $i$. This collection is then returned  to the urn.  
 We then continue the process, each time taking an independent (of everything else) copy of the replacement matrix $L$. Note that the described model involves replacement, meaning that the picked ball is placed back into the urn after each draw. However, it is also possible to study the model without replacement by considering the replacement matrix $L-I$, where $I$ is the $J\times J$ identity matrix. In this case, it might happen that some diagonal entries of this new matrix are equal to -1, and in this case a ball is removed from the urn. The \emph{urn process} is the sequence $(B(n))_{n\geq 1}$ of 
$J$-dimensional random vectors with nonnegative integer coordinates, and the $j$-th coordinate $B_j(n)$ of $B(n)$ represents the number of balls of color $j$ in the urn after the $n$-th draw, for  $j\in [J]$. We also define $B^{\circ}(n)$ to be the number of drawn balls up to time $n$, that is, $B^\circ_j(n)$ represents the number of drawn balls of color $j$ up to the $n$th draw (in particular, $B^\circ_1(n)+\cdots+B^\circ_J(n)=n$).  As one expects, the limit behavior of $(B(n))_{n\geq 1}$ and $(B^\circ(n))_{n\ge 1}$ depends on the distribution of the replacement matrix $L$, and in particular on the spectral properties of its mean value matrix $A$. 

The literature on limit theorems for P\'olya urn models is enormous and any attempt to give a complete survey here is hopeless, but we mention some relevant references and results in this direction. For additional results, the reader is referred to the cited articles and the references therein.  In 1930, in his original
article \cite{polya-1930}, P\'olya investigates a two-color urn process with replacement matrix $L$ being the identity.
If $L$ is a non-random, irreducible matrix with exclusively non-negative entries, then it is well-established that the sequence $B^\circ(n)/n$ converges almost surely to  $\mathrm{u}$ as $n$ goes to infinity, where $\mathrm{u}$ is  the left eigenvector associated with $\rho$, the spectral radius of $A=\E L$. The coordinates of $\mathrm{u}$ are all non-negative and normalized in such a way that they sum up to one; see  \cite{athreya-karlin-1969,janson-fclt-urns,smythe-clt-urns-1996,Mueller+Neininger:2018} for more details. The second-order behavior of the sequence $(B^\circ(n))_{n\in\N}$ depends on the second eigenvalue $\lambda_2$ (ordered by real parts) of $A$. If $\mathrm{Re}(\lambda_2)$, the real part of the second-largest eigenvalue, is less than or equal to $\rho/2$, then the fluctuations around the limit $\mathrm{u}$ are Gaussian (with a random variance). The  magnitude of the fluctuations is of order $\sqrt{n}$ when $\mathrm{Re}(\lambda_2)<\rho/2$ and of order $\log(n)\sqrt{n}$ in  the critical case $\mathrm{Re}(\lambda_2) = \rho/2$. Conversely, if $\mathrm{Re}(\lambda_2) > \rho/2$, then the fluctuations are non-Gaussian and of higher order. See Janson \cite{janson-fclt-urns} for this  trichotomy  and  \cite{pouyanne-polya} for an approach based on the spectral decomposition of a suitable finite difference transition operator on polynomial functions.

Apart from these seminal results, the model of P\'olya urns has been extended and more precise asymptotics are known.
Several generalizations have been considered in \cite{janson-fclt-urns}. Another possible extension is to consider measure valued P\'olya processes; see the recent work \cite{janson-mailler-2021} for second order asymptotics of such processes for infinitely many colors and the cited literature there for additional results.

\textbf{The model and our contribution.}
%In the current paper we provide an application of the  discrete time multitype Crump-Mode-Jagers (CMJ) process $(Z_n^{\Phi})_{n\in \N}$ counted with a characteristic $\Phi$, as introduced in  \cite{kolesko-sava-huss-char} to a certain modification of a Pólya urn model with two interacting urns. 
In the current paper we consider a modification of the P\'olya urn model containing two urns marked $\mathsf{U}_1$ and $\mathsf{U}_2$. For fixed  $J\in\N$ to be the number of colors we consider the random $J\times J$ matrix $L$ as above, with independent column vectors $L^{(j)}$ and expectation matrix $A=\E L$. With these initial conditions, we define the $\N^J$-valued stochastic process $(B(n))_{n\in \N}$, with $B(n)=(B_1(n),\ldots,B_J(n))$ as following. Suppose that at time $0$ we have one ball of type (color) $j_0$ in urn $\mathsf{U}_1$. We pick it out and we put a collection of $L^{(j_0)}=(L^{(j_0,1)},\dots, L^{(j_0,J)})$ balls  of types $1,2,\ldots,J$ respectively into urn $\mathsf{U}_2$;  that is, for any $i\in [J]$, we put $L^{(j_0,i)}$ balls of type $i$ into urn $\mathsf{U}_2$, thus we have $\sum_{i=1}^JL^{(j_0,i)}$ balls in urn $\mathsf{U}_2$.
%We emphasize again that the \textcolor{blue}{drawing is without replacement so the picked ball is not returned to the other urn}.
At the next step, we draw from urn $\mathsf{U}_2$ balls uniformly at random, one after another, and for any picked ball of type $j$ we add an independent collection of balls with distribution  $L^{(j)}$ into urn $\mathsf{U}_1$. We continue until urn $\mathsf{U}_2$ is empty and then we exchange the roles of urns $\mathsf{U}_1$ and $\mathsf{U}_2$. We emphasize  that, contrary to the P\'olya urn model, in the two urns model it is more convenient to consider the drawing without replacement, that is, each picked ball is not returned to either urn, but only determines the future additions of balls. In particular, all coefficients of $L$ are non-negative.
For $j\in [J]$,  by $B_j(n)$ we denote  the total number of balls of type $j$ that have been added to one of these two urns up to (and including) the $n$-th draw.
 
Graphically, we can draw a random tree in order to visualize the step-by-step evacuation of one urn and the refillment of the other one as following: we colour the nodes of the tree in colours $\{1,\ldots,J\}$, the content of the urn $\mathsf{U}_1$ represents the root of the tree coloured with some fixed $j_0\in \{1,\ldots,J\}$, i.e.~the zero generation; after one draw of the node of type $j_0$ the content $L^{(j_0)}$ of $\mathsf{U}_2$ represents the first generation. Then after choosing balls (i.e.~nodes of the tree at level one) uniformly at random without replacement and putting their offspring in the other urn $\mathsf{U}_1$, we create step-by-step the second generation of the tree, and fill up step-by-step $\mathsf{U}_1$ again. Thus what we propose here is a more refined branching process where the transition from generation $k$ to generation $k+1$ is looked at after each member of generation $k$ reproduces. Visualizing the process as a random tree that grows after each node is chosen, the quantity $B_j(n)$ represents the number of nodes of type $j$ and $\sum_{j=1}^JB_j(n)$ represents the total number of nodes in the tree after $n$ steps of our process, that is after $n$ balls have been picked up from $\mathsf{U}_1$ and $\mathsf{U}_2$. For a better understanding, we illustrate this process on an example in Appendix \ref{sec:apA} and Figure \ref{fig:urn}.

The main focus of the current work is to investigate first and second order asymptotics of $B_j(n)$ as $n\to\infty$, $j\in [J]$.
It may happen that with positive probability $L^{(j,i)}$ vanishes for all $i\in[J]$. In such a case we do not add any new ball to the urn and we just remove the picked ball of type $j$. In particular, it can happen that after a finite number $n_0$ of steps,  both  urns are empty and in such a case we define $B(n)=B(n_0)$, for $n\ge n_0$. Since we are interested in the long time behavior of the urn process, we restrict the analysis to a set where this does not happen, i.e.~to the survival event $\mathcal S=\{|B(n)|\to\infty\}$.

We can also define the corresponding sequence $(B^\circ(n))_{n\ge 0}$ that represents the types of the drawn balls up to time $n$. While it is possible to ask about limit theorems for $B^\circ(n)$, the method developed in this paper for studying $B(n)$ is directly applicable to $B^\circ(n)$. Therefore, we focus our attention exclusively on the sequence $(B(n))_{n\in\N}$.
 
The approach we use to investigate $(B_j(n))_{n\geq 0}$, $j\in [J]$ is to embed it into a multitype discrete time branching process $(Z_n)_{n\in\N}$ with offspring distribution matrix  $L$. A similar approach using the Athreya-Karlin embedding allowed Janson \cite{janson-fclt-urns} to study $(B^\circ_j(n))_{n\geq 0}$ for the P\'olya urn model.
One difference between our model and the one in \cite{janson-fclt-urns} is that, in the latter the process is embedded  into a multitype continuous time Markov branching process, and an individual reproduces after an exponential clock rings. In our model, in the embedded branching process, an individual reproduces after  deterministic time 1. The lattice nature of the model manifests itself in the second order behavior of $(B_{j}(n))_{n\geq 0}$. 
\vspace{-0.3cm}
\paragraph{Assumptions.} In this work we use the following assumptions:
\vspace{-0.3cm}
\begin{enumerate}[(GW1)]\setlength\itemsep{0em}
\item $A$ has spectral radius $\rho>1$.
\item $A$ is positively regular.
\item $\mathbf{0}\neq \sum_{j=1}^J\Cov\left[L^{(j)}\right]$ and  $\Var[L^{(i,j)}]<\infty$ for all $i,j\in[J]$.
\item For every $i,j\in[J]$, the expectation  $\E[ L^{(i,j)}\log L^{(i,j)}]$ is finite.
\end{enumerate}
In the third condition above, $\mathbf{0}$ is the $J\times J$ zero matrix, and for $j\in[J], \Cov\left[L^{(j)}\right]$ represents the $J \times J$ covariance matrix of the vector $L^{(j)}$. If the matrix $A$ is irreducible, Perron-Frobenius theorem ensures that the dominant eigenvalue $\rho$ of $A$ is real, positive and simple.  If $\rho>1$, this means that the multitype branching process with offspring distribution matrix $L$ is supercritical, i.e.~$\Prob(\mathcal S)>0$. If $\mathrm u$ is the corresponding eigenvector for $\rho$, then clearly, all the entries $\mathrm u_j$ are strictly greater than zero for any $j\in[J]$ and we assume that $\mathrm u$ is normalized in such a way that $\sum_{j\in[J]}\mathrm u_j=1$.
First order asymptotics of $B_j(n)$, $j\in[J]$ are determined by $\rho$ and the vector $\mathrm u$. 
 \begin{theorem}\label{thm:main_SSLN}
Assume (GW1),(GW2) and (GW4) hold. Then for any $j\in [J]$ we have the following strong law of large numbers for the total number of balls of type $j$ after $n$  draws:
 \begin{align*}
 \lim_{n\to\infty}\frac{B_j(n)}{n}=\rho \mathrm u_j,\quad \quad \ProbS\textit{-almost surely}. 
 \end{align*}
 \end{theorem} 
Thus, the first order behavior of $B_j(n)$, $j\in [J]$ resembles the first order behavior of $B^\circ_j(n)$ from the model with one urn \cite{janson-fclt-urns}. 
In order to understand second order asymptotics of $B_j(n)$ we need full information on the spectral decomposition of the mean replacement matrix $A$. We denote by
$\sigma_A$ the spectrum of the matrix $A$ and define $\sigma_A^1=\{\lambda\in\sigma_A:\ |\lambda|>\sqrt{\rho}\}$, $\sigma_A^2=\{\lambda\in\sigma_A:\ |\lambda|=\sqrt{\rho}\}$, and $\sigma_A^3=\{\lambda\in\sigma_A:\ |\lambda|<\sqrt{\rho}\}$. Then we can write
$$\sigma_A=\sigma_A^{1}\cup\sigma_A^{2}\cup\sigma_A^{3}.$$ 
For a simple eigenvalue $\lambda\in\sigma_A$, we denote by $\mathrm u^\lambda$ and $\mathrm v^\lambda$ the corresponding left and right eigenvectors. We set
$$\gamma=\max\{|\lambda|:\lambda\in\sigma_A\setminus\{\rho\}\}\quad \text{and}\quad \Gamma=\{\lambda\in\sigma_A:|\lambda|=\gamma\},$$
 so $\rho-\gamma$ is the \textit{spectral gap} of $A$, and $\Gamma$ is the set of eigenvalues of $A$ where the spectral gap is achieved. For a complex number $z$ we set $\log_\rho z=\frac{\log z}{\log \rho},$ where $z\mapsto\log z$ is the principal branch of the natural logarithm. Denote by $\{\mathrm e_j\}_{j\in[J]}$ the standard basis vectors in $\R^J$. For $j\in [J]$, consider the following row vector in $\R^{1\times J}$
 \begin{equation}\label{eq:w_j}
 \mathrm{w}_j=\mathrm{e}_j^{\top}A-\rho\mathrm{u}_j\mathrm{1},
 \end{equation} 
 whose $i$-th entry is $\mathrm{w}_{ji}=a_{ji}-\rho\mathrm{u}_j=\E[L^{(i,j)}]-\rho\mathrm{u}_j$, for $1\leq i\leq J$.
 In the above equation, $\mathrm{1}$ denotes the vector in $\mathbb{R}^{1\times J}$ with all entries equal to one.
Our main result provides second order behavior of $(B_j(n))_{n\in \N}$.

\begin{theorem}\label{them:main_CLT}
Assume that (GW1)-(GW3) hold, all $\lambda\in\Gamma$ are simple, and there exists $\delta>0$ such that $\E\big[\big(L^{(i,j)}\big)^{2+\delta}\big]<\infty$ for all $i,j\in[J]$. Then for any $j\in[J]$ the following trichotomy holds:
\vspace{-0.25cm}
\begin{enumerate}[i)]\setlength\itemsep{-0.25em}
\item If $\gamma>\sqrt \rho$, then for any $\lambda\in\Gamma$ there exists a 1-periodic, continuous function $f_\lambda:\R\to\C$ and  random variables $X,X_\lambda$, such that the following holds:
$$B_j(n)=\rho \mathrm u_j \cdot n+\sum_{\lambda\in \Gamma}n^{\log_\rho\lambda}f_\lambda(\log_\rho n-X)X_\lambda+o_{\Prob}\big(n^{\log_\rho\gamma}\big).$$
\item If $\gamma = \sqrt \rho$ and for some $\lambda\in\sigma^{2}_A$ and  $i\in[J]$ we have  $\mathrm{w}_j\mathrm u^\lambda\neq0$ and $\Var[\mathrm v^\lambda L\mathrm{e}_i]>0$ for $\mathrm{w}_j$ defined as in \eqref{eq:w_j}, then there exists a 1-periodic, continuous function $\Uppsi:\R\to(0,\infty)$ and a random variable $X$  such that, conditionally on $\S$, the following convergence in distribution holds:
$$\frac{B_j(n)-\rho \mathrm u_j\cdot n}{\sqrt {n\log_{\rho} n}\ \Uppsi(\log_\rho n-X)}\distto \mathcal{N}(0,1),\qquad \text{as }n\to\infty.$$
\item If $\gamma < \sqrt \rho$ and for some $\lambda\in\sigma^{3}_A$ and some $i\in[J]$ we have  $\mathrm{w}_j\mathrm u^\lambda\neq0$ and $\Var[\mathrm v^\lambda L\mathrm{e}_i]>0$ with $\mathrm{w}_j$ defined as in \eqref{eq:w_j}, then there exists a 1-periodic, continuous function $\Uppsi:\R\to(0,\infty)$ and a random variable $X$  such that,
conditionally on $\S$, the following convergence in distribution holds:
		$$\frac{B_j(n)-\rho \mathrm u_j\cdot n}{\sqrt n\ \Uppsi(\log_\rho n-X)}\distto \mathcal{N}(0,1),\qquad \text{as }n\to\infty.$$
\end{enumerate} 
%	In all the three cases, the random variable $X$ is given by $X:=\log_{\gamma} W|\mathrm{u}|$, $W$ is the scalar random variable from Kesten-Stigum Theorem and $u$ is the right eigenvector of $A$, with $|\mathrm{u}|=\sum_{j=1}^J\mathrm{u}^j.$
\end{theorem}
A more general and  quantified result where the periodic functions are explicitly defined is provided in Theorem \ref{thm:general_limit_theorem}.

Note that the above result slightly differs from the one urn model where the functions $\Uppsi$ and $f_\lambda$ are actually constants. What might be even more surprising is that, in our model, the phase transition occurs at $\sqrt{\rho}$ rather than at $\rho/2$ as observed in the P\'olya urn model. The heuristic explanation is as follows:  the growth in mean of the corresponding continuous-time branching process is driven by the semigroup $e^{tA}$. In particular, the leading asymptotic is $e^{t\rho}$ and the next order is $|e^{t \lambda_2}|=e^{t \mathrm{Re}\lambda_2}$. We anticipate observing Gaussian fluctuations in the branching process at the scale $\sqrt{e^{t\rho}}$ (with possible polynomial corrections), providing a natural threshold for $\mathrm{Re}\lambda_2$ in relation to $\rho/2$. On the other hand, the two-urn model is embedded into a discrete-time branching process whose growth in the mean is driven by the semigroup $A^n$ (or $(I+A)^n$ in the model with no replacement). Thus, the leading term is at scale $\rho^n$ and the subleading term at scale $|\lambda_2|^n$. As before, we expect to observe Gaussian fluctuations at scale $\sqrt{\rho^n}$, which induce natural distinctions depending on the relative locations of $|\lambda_2|$ and $\sqrt{\rho}$.

\textbf{Structure of the paper.} In Section \ref{sec:prelim} we introduce multitype branching processes $(Z_n)_{n\in\N}$ and Crump-Mode-Jagers processes $(Z_n^{\Phi})_{n\in\N}$ counted with a characteristic $\Phi:\Z\to \R^{1\times J}$. Then in Section \ref{sec:embed} we show how to relate our model $(B_j(n))$ with two interacting urns with a branching process $(Z^{\Phi^j}_n)_{n\in \N}$ counted with a characteristic $\Phi^j$. By applying  \cite[Proposition 4.1]{kolesko-sava-huss-char} to $(Z^{\Phi^j}_n)_{n\in\N}$, we then obtain first order asymptotics of $(B_j(n))_{n\in\N}$ in Theorem \ref{thm:main_SSLN} for any $j\in[J]$. The main result is proved in Section \ref{sec:main}. We conclude with  Appendix \ref{sec:apA} where the model with two interacting urns is illustrated on an example with deterministic replacement matrix, and  Appendix \ref{sec:apB} where higher moments estimates for $(Z_n^{\Phi})$ are given, for general characteristics $\Phi$. 

\section{Preliminaries}\label{sec:prelim}

For the rest of the paper $(\Omega,\mathcal{A},\Prob)$ is a probability space on which all the random variables and processes we consider are defined. We write $\AS$ for almost sure convergence, $\Probto$ for convergence in probability, $\distto$ for convergence in distribution and $\stabto$ for stable convergence (c.f. \cite{aldous-eagleson} for the definition and properties). We also use $\ProbS$ to denote the conditional probability $\Prob[\cdot|\S]$ and the corresponding convergences are denoted by $\ProbSto,\distSto,\stabSto$. We use the notations $\N =\{1,2,\ldots,\}$ and $\N_0 =\{0,1,2,\ldots,\}$.

\textbf{Stochastic processes.} Our convergence results of stochastic processes use the usual space $\mathcal{D}$ of right-continuous functions with left-hand limits, always equipped with the Skorokhod $\mathbf{J}_1$ topology. For a finite-dimensional vector space $E$ and any interval $J\subseteq [-\infty,\infty]$, we denote by $\mathcal{D}(J)=\mathcal{D}(J,E)$ the space of all right continuous functions from $J$ to $E$ with left-hand limits.

For a $n\times m$ matrix $A=(a_{ij})_{i,j}$ with $m,n\in\N$, the Hilbert–Schmidt norm of $A$, called also Frobenius norm, is defined as
\vspace{-0.25cm}
 $$\|A\|_{HS}=\Big(\sum_{i=1}^n\sum_{j=1}	^m|a_{ij}|^2\Big)^{1/2}.$$
Since for any vector its Hilbert-Schmidt norm coincides with its Euclidean norm,  for the rest of the paper we only write $\|\cdot\|$ instead of $\|\cdot\|_{HS}$. 

%\paragraph*{Stable convergence.} 
%A sequence $(Y_n)_{n\in\N}$ of random variables defined on $(\Omega,\mathcal F,\Prob)$ \emph{converges stably in distribution} to a random variable $Y$ defined on an extension of $(\Omega,\mathcal F,\Prob)$, written  $$Y_n\stabto Y,$$ if  for any bounded continuous function $f:\R\to\R$ and for any $\mathcal{F}$-measurable bounded random variable $Z$, it holds:
%$$\lim_{n\to\infty}\E[f(Y_n)Z]=\E[f(Y)Z].$$
%Alternatively (cf. \cite[Proposition 1.]{aldous-eagleson}), for any fixed $\mathcal F$-measurable random variable $X$, it holds $$(X,Y_n)\distto(X,Y), \quad \text{as }n\to\infty.$$ 

\textbf{Ulam-Harris tree $\U$.}
An Ulam-Harris  tree $\U$ is an infinite rooted tree with vertex set $V_{\infty}= \bigcup_{n \in \N_0} \N^n$, the set of all finite strings or words $i_1\cdots i_n$ of positive integers over $n$ letters, including the empty word $\varnothing$ which we take to be the root, and with an edge joining $i_1\cdots i_n$ and $i_1\cdots i_{n+1}$ for any $n\in \N_0$ and any $i_1, \cdots, i_{n+1}\in \N$. 
Thus every vertex $v=i_1\cdots i_n$ has outdegree $\infty$, and the children of $v$ are the words $v1,v2,\ldots$ and we let them have this order so that $\U$ becomes an infinite ordered rooted tree. We will identify $\U$ with its vertex set $V_{\infty}$, when no confusion arises. For vertices $v=i_1\cdots i_n$ we also write $v=(i_1,\ldots,i_n)$, and if $u=(j_1,\ldots,j_m)$ we write $uv$ for the concatenation of the words $u$ and $v$, that is  $uv=(j_1,\ldots,j_m,i_1,\ldots,i_n)$. The parent of $i_1\cdots i_n$ is $i_1\cdots i_{n-1}$.
%Further, if $k \leq m$, we set $u|_k \defeq u_1 \ldots u_k$ for the vertex $u$ truncated at height $k$.
Finally, for $u \in \U$, we use the notation $|u|=n$ for $u \in\N^n$, i.e.~$u$ is a word of length (or height) $n$, that is, at distance $n$ from the root $\varnothing$. 
%The family of ordered rooted trees can be identified with the set of all subtrees $\T_u$ of $\U$ rooted at $u\in V_{\infty}$, that have the property that for all $v\in V_{\T_u}$,
%$vi\in V_{\T_u}\  $  implies $ \ vj\in V_{\T_j}, \quad \text{for all }j\leq i$, where as usual $V_{\T_u}$ denotes the vertex set of the tree $\T_u$ (the tree rooted at $u$ together with all infinite paths going away from it) which is identified with $\T_u$ itself when no confusion arises.
For any $u\in V_{\infty}$, by $\T_u$ we mean the subtree of $\U$ rooted at $u$, that is $u$ together with all infinite paths going away from $u$, and for  $u,v\in \U$ we denote by $d(u,v)$ their graph distance.
% that is, the length of the shortest path between $u$ and $v$.
For trees rooted at $\varnothing$, we omit the root and we write only $\T$.
%Finally, by identifying $\T$ with its vertex set $V(\T)$, we will regard $\mathcal{T}$ as the family of all rooted (not rooted necessarly at $\varnothing$) subtrees $\T$ of $\U$ that satisfy:
%\begin{align*}
%& \varnothing  \in \T,\\
 %& v_1\cdots v_{n+1} \in \T\  \Rightarrow \  v_1\cdots v_{n}\in \T\\
 %& v_1\cdots v_ni  \in \T\  \Rightarrow \ v_1\cdots v_nj\in \T, \quad \forall \ j\leq i.\\
%\end{align*}
%Denote by $\mathcal{T}$ the set of all ordered rooted subtrees $\T$  of $\U$. 
For $J\in \N$, a $J$-type tree is a pair $(\T,\type)$ where $\T$ is a rooted subtree of $\U$ and $\type:\T\to\{1,\ldots,J\}$ is a function defined on the vertices of $\T$ that returns for each vertex $v$ its type $\type(v)$. %We will often omit the subindex $\T$ in the type function $\type$ if it is clear from the context the tree we are referring to. 
%We denote by $\mathcal{T}^{[J]}$ the set of all $J$-type trees, and elements of $\mathcal{T}^{[J]}$ is referred to as  $\T$ without explicitly mentioning the type function $\type$.

\textbf{Multitype branching processes.} Consider the random $J\times J$ matrix $L$ with independent column vectors $L^{(j)}$, for $1\leq j\leq J$ as in the introduction.
%The random variable $L^{(j)}$ is the offspring distribution vector produced by a particle of type $j$ and the  matrix $L^{\top}=(L^{(i,j)})_{i,j\in [J]}$ is the offspring distribution matrix, where for  $i,j\in [J]$ the random variable $L^{(i,j)}$ is the number of offspring of type $j$ produced by a particle of type $i$.
Multitype Galton-Watson trees are random variables taking values in the set of $J$-type trees $(\T,\type)$, where the type function $\type$ is random and defined in terms of the matrix $L$. Let $(L(u))_{u\in \U}$ be a family of i.i.d~copies of $L$ indexed over the vertices of $\U$.
For any $i\in [J]$, we define the random labeled tree $\T^{i}$ rooted at $\varnothing$, with the associated type function
$\type=\type^{i}:\T^{i}\to \{1,\ldots,J\}$ defined 
recursively as follows: $\varnothing\in\T^{i}$ and  $\type(\varnothing)=i$.
Now suppose that $u=j_1\ldots j_m\in\T^{i}$ with $\type(u)=j$, for some $j\in[J]$. Then 
$$j_1 \ldots j_m k\in \T^{i} \quad\text{iff}\quad k\le L^{(j,1)}(u)+\dots+L^{(j,J)}(u),$$ and we set
$\mathrm \type(u k)=\ell$ whenever 
$$L^{(j,1)}(u)+\dots+L^{(j,\ell-1)}(u)<k\le L^{(j,1)}(u)+\dots+L^{(j,\ell)}(u).$$

The multitype branching process  $Z_n=(Z^1_n, \dots , Z^J_n)$ associated with the pair $(\T^{i_0},\type)$, and starting from a single particle (or individual) of type $i_0\in [ J]$ at the root $\varnothing$, that is $\type(\varnothing)=i_0$, is defined as: $Z_0=\mathrm{e}_{i_0} $ and for $n\geq 1$
\begin{align*}
Z^i_n=\#\{u\in\T^{i_0}:|u|=n\text{ and }\type(u)=i\}, \quad \text{for } i\in [J],
\end{align*}
so $Z_n^i$ represents the number of individuals of type $i$ in the $n$-th generation, or  the number of vertices $u\in \T^{i_0}$ with $|u|=n$ and $\type(u)=i$.
The main results of \cite{kolesko-sava-huss-char} that we use in the current paper, hold under assumptions (GW1)-(GW3) on $(Z_n)_{n\in \N}$, that we suppose to hold here as well. In particular, $(Z_n)_{n\in\N}$ is a supercritical branching process. %Under assumptions (GW1, (GW2) and (GW4), Kesten-Stigum theorem  \cite{kesten1966} implies the existence of a scalar random variable $W$ such that 
%$\lim_{n\to\infty}\rho^{-n}Z_n=W\mathrm{u}$ almost surely.

\textbf{Spectral decomposition of $A$}. Recall the decomposition of the spectrum $\sigma_A=\sigma_A^{1}\cup\sigma_A^{2}\cup\sigma_A^{3}$ of the matrix $A$.
From the Jordan-Chevalley decomposition (which is unique up to the order of the Jordan blocks) of $A$ we infer the existence of projections $(\pi_\lambda)_{\lambda\in \sigma_A}$
that commute with $A$ and satisfy $\sum_{\lambda\in\sigma_A}\pi_{\lambda}=I$
and 
$$A\pi_{\lambda}=\pi_{\lambda}A=\lambda\pi_{\lambda}+N_{\lambda},$$
where $N_{\lambda}=\pi_{\lambda}N_{\lambda}=N_{\lambda}\pi_{\lambda}$ is a nilpotent matrix. Moreover,
for any $\lambda_1,\lambda_2\in\sigma_A$ it holds  $\pi_{\lambda_1}\pi_{\lambda_2}=\pi_{\lambda_1}\ind{\lambda_1=\lambda_2}$. If $\lambda\in\sigma_A$ is a simple eigenvalue of $A$ and $\mathrm u^\lambda, \mathrm v^\lambda$ are the corresponding left and right eigenvectors normalized in such a way that $\mathrm v^\lambda\mathrm u^\lambda=1$,  then 
$\pi_\lambda=\mathrm u^\lambda\mathrm v^\lambda$.
 If we write $N=\sum_{\lambda\in\sigma_A}N_{\lambda}$, then $N$ is also a nilpotent matrix and we have $N\pi_{\lambda}=N_{\lambda}$. Thus $A$ can be decomposed into the semisimple part $D=\sum_{\lambda\in\sigma_A}\lambda\pi_{\lambda}$ and the nilpotent part $N$ as  $A=D+N$.

For any $\lambda\in\sigma_A$, we denote by $d_{\lambda}\geq 0$ the integer such that $N_{\lambda}^{d_{\lambda}}\neq 0$
but $N_{\lambda}^{d_{\lambda}+1}=0$ (hence $d_{\lambda}+1$ is at most  the multiplicity of $\lambda$). So $d_{\lambda}=0$ if and only if $N_{\lambda}=0$, and this happens for all $\lambda$ if and only if $A$ is diagonalizable, that is $A$ has a complete set of $J$ independent eigenvectors. Since $\rho$ is a simple eigenvalue, we have $N_{\rho}=0$ and $d_{\rho}=0$, and $\pi_{\rho}=\mathrm u\mathrm v$.
We set
\begin{align*}
\pi^{(1)}=\sum_{\lambda\in\sigma^1_A}\pi_\lambda,\quad
\pi^{(2)}=\sum_{\lambda\in\sigma^2_A}\pi_\lambda,\quad
\pi^{(3)}=\sum_{\lambda\in\sigma^3_A}\pi_\lambda,
\end{align*}
and for  $i=1,2$, we define
$$A_i= A\pi^{(i)}+\big(I-\pi^{(i)}\big).$$
The process  $\big(W_n^{(i)}\big)_{n\in \N}$ defined by 
$$W_n^{(i)}= A_i^{-n}\pi^{(i)} Z_n$$ is a $\A_n$-martingale,
where $(\A_n)_{n\ge0}$ is the filtration $\A_n=\sigma(\{L(u):|u|\le n\})$.
According to \cite[Lemma 2.2]{kolesko-sava-huss-char}, $W_n^{(1)}$ converges in $\mathcal{L}^2(\Omega,\A,\Prob)$ to a random variable $W^{(1)}$ whose expectation is $\E W^{(1)}=\pi^{(1)}\mathrm e_{i_0}$. In particular, we have the convergence
$$\rho^{-n} Z_n\to \pi_\rho W^{(1)} = W \mathrm u,\quad \text{as } n\to\infty$$
in $\mathcal{L}^2$, and the random variable $W$ is given by $W= \mathrm v \cdot  W^{(1)}$ with $\E W=\mathrm v^{i_0}>0$.
The classical Kesten-Stigum theorem \cite{kesten1966} states that under (GW1), (GW2) and (GW4) the convergence above holds almost surely.  For the rest of the paper, when we use the random variable $W$, we always mean the limit random variable from the Kesten-Stigum theorem.

\subsection{Branching processes counted with a characteristic}

We recall that a characteristic of dimension one is a random function $\Phi:\Z\to\R^{1\times J}$, so for each $k\in \Z$, $\Phi(k)$ is a $\R^{1\times J}$-valued random variable defined on the same probability space $(\Omega,\A,\Pb)$ where the Galton-Watson process $(Z_n)_{n\in\N}$ and its genealogical tree $\T$ are defined. A {\it deterministic characteristic} is just a fixed function $\Phi:\Z\to\R^{1\times J}$. For a random function $\Phi:\Z\to \R^{1\times J}$ and the multitype Galton-Watson tree $\T$, the process $(Z^\Phi_n)_{n\in \N}$ which for any $n\in\N$ is defined as 
\begin{align*}
Z^\Phi_n=\sum_{u\in\T}\Phi_u(n-|u|)\mathrm{e}_{\type(u)}
\end{align*}
is called {\it multitype Crump-Mode-Jagers (shortly CMJ) process counted with characteristic $\Phi$} or simply {\it branching process counted with characteristic $\Phi$}, where $\left(\Phi_u\right)_{u\in\U}$ is an i.i.d.~copy of $\Phi$. First and second order asymptotics for $(Z^\Phi_n)_{n\in \N}$, under mild assumptions on $\Phi$ have been considered in \cite[Proposition 4.1]{kolesko-sava-huss-char} and in \cite[Theorem 3.5]{kolesko-sava-huss-char}, respectively. We use these two results below for a particular choice of the characteristic $\Phi$, and show how the branching process with this particular choice  of the characteristic can be related to the two urn model with alternating urns $\mathsf{U}_1$ and $\mathsf{U}_2$. 

\textbf{The choice of the characteristic.}
Let $U\sim\mathsf{Unif}[0,1]$ be an uniform random variable taking values in $[0,1]$, and defined on the probability space $(\Omega,\A,\mathbb{P})$. For every  threshold  $x\in[0,1)$ we define the characteristic $\Phi^{\mathsf{t}}_{x}:\N_0\to\R^{1\times J}$ by
\begin{align} \label{eq:char-urn-all}
\Phi_x^{\mathsf{t}}(k)\mathrm{e}_i=\ind{{k\geq 1}}+\ind{k=0}\ind{U \le x}, \quad \text{for } 1\leq i\leq J,
\end{align}
or, in a simplified way, for  $1\leq i\leq J$:
\begin{equation*}
\Phi_x^{\mathsf{t}}(k)\mathrm{e}_i=
\begin{cases}
1,&\quad \text{for } k\geq 1\\
1  \text{ with probability } \Prob(U\leq x)=x,& \quad \text{for } k=0 \\
0 \text{ with probability } \Prob(U> x)=1-x,& \quad \text{for } k=0 
\end{cases}.
\end{equation*}
Similarly, for $j\in[J]$, we define $\Phi_{x}^j:\N_0\to\R^{1\times J}$ by
\begin{align}
\label{eq:char-urn-type}
\Phi_x^j(k)\mathrm{e}_i= \ind{k\geq 1}\underbrace{\langle\mathrm{e}_j,L\mathrm{e}_i\rangle}_{=L^{(i,j)}}+\ind{k=0}\ind{U \le x}\langle\mathrm{e}_j,L\mathrm{e}_i\rangle, \quad \text{for } 1\leq i\leq J,
\end{align}
so $\Phi_x^j(k)\mathrm{e}_i=\Phi^{\mathsf{t}}_x(k)\mathrm{e}_i L^{(i,j)}$.
For the uniform random variable $U$ on $[0,1]$, let $(U_u)_{u\in\U}$ be an i.i.d.~copy of $U$. 
For $u\in \U$ let now $ \Phi^{\mathsf{t}}_{x,u}$ (respectively $ \Phi^j_{x,u}$) be defined through \eqref{eq:char-urn-all} (respectively \eqref{eq:char-urn-type}) with $U,L$  being replaced by  $U_u,\ L(u)$. Then  $\big((L(u),\Phi^{\mathsf{t}}_{x,u},\Phi_{x,u}^j)\big)_{u\in \U}$ is an i.i.d.~collection of copies of $(L,\Phi^{\mathsf{t}}_x,\Phi_x^j)$, for $j\in [J]$.

For the characteristic $\Phi_x^{\mathsf{t}}$ from \eqref{eq:char-urn-all}, threshold $x\in [0,1)$, and multitype Galton-Watson tree $\T$, the process $\big(Z^{\Phi_x^{\mathsf{t}}}_n\big)_{n\geq 0}$ counts then the total number of individuals $u\in \U$ born before time $n$ (including the root), and those born at time $n$ with $U_u\le x$ since we have
\begin{align*}
Z_n^{\Phi^{\mathsf{t}}_x}=\sum_{u\in \T}\Phi^{\mathsf{t}}_{x,u}(n-|u|)\mathrm{e}_{\type(u)}=\sum_{k=0}^{n-1}\sum_{u\in \T;  |u|=k}1+\sum_{u\in \T; \ |u|=n}\ind{U_u\leq x}
=\sum_{k=0}^{n-1}|Z_k|+\sum_{u\in \T; \ |u|=n}\ind{U_u\leq x},
\end{align*}
where $|Z_n|=\sum_{j=1}^JZ_n^j$ represents the size of the $n$-th generation of the Galton-Watson process.
On the other hand $L\mathrm e_i=L^{(i)}$ represents the random collection of individuals born from an individual of type $i$, while $L^{(i,j)}=\langle\mathrm{e}_j,L\mathrm{e}_i\rangle$ represents the random number of offspring of type $j$ of an individual of type $i$ for $i,j\in[J]$. Therefore $Z_n^{\Phi_x^j}$ counts the number of individuals of type $j$ born up to time $n$, and those of type $j$ born at time $n+1$ but with threshold $\le x$, so $Z_n^{\Phi^j_x}$ can be written as
\begin{align*}
Z_n^{\Phi^j_x}&=\sum_{k=0}^{n-1}\sum_{u\in\T;  |u|=k}\langle\mathrm{e}_j,L(u)\mathrm{e}_{\type(u)}\rangle+\sum_{u\in \T;|u|=n+1,\type(u)=j}\ind{U_u\leq x},
\end{align*}
since $Z_{k+1}^j=\sum_{u\in \T;|u|=k}\langle\mathrm{e}_j,L(u)\mathrm{e}_{\type(u)}\rangle$ represents the number of offspring of type $j$ in the $(k+1)-st$ generation and $|\{u\in \T;|u|=n+1,\type(u)=j\}|=Z_{n+1}^j$.
Summing up over all $j\in [J]$ gives 
$$Z_{n+1}^{\Phi^{\mathsf{t}}_x}-1= \sum_{j=1}^J Z_n^{\Phi^j_x},$$
and an application of \cite[Proposition 4.1]{kolesko-sava-huss-char} to $Z_n^{\Phi^{\mathsf{t}}_x}$ and $Z_n^{\Phi^j_x}$, for $j\in[J]$ yields the law of large numbers.
\begin{proposition}\label{prop:lln1}Under assumptions (GW1),(GW2), and (GW4) for any threshold $x\in[0,1)$ and characteristic $\Phi_x^{\mathsf{t}}$ (respectively $\Phi_x^j$, $j\in[J]$) defined in \eqref{eq:char-urn-all} (respectively \eqref{eq:char-urn-type}) we have:
\begin{align*}
\lim_{n\to\infty}\dfrac{Z_n^{\Phi^{\mathsf{t}}_x}}{\rho^n}&=\Big(\frac{1}{\rho-1}+x\Big)W,\quad \ProbS-\text{almost surely},\\
\lim_{n\to\infty}\dfrac{Z_n^{\Phi^j_x}}{\rho^n}&=\Big(\frac{1}{\rho-1}+x\Big)W\rho\mathrm  u_j,\quad \ProbS-\text{almost surely}.
\end{align*}
\end{proposition}
\begin{proof}
Since for any fixed  $x\in[0,1)$ the random variables $(\ind{U_u\leq x})_{u\in\U}$ are i.i.d.~Bernoulli $\mathsf{Bern}(x)$, in view of the strong law of large numbers we obtain
$$\lim_{n\to\infty}\frac{1}{|Z_n|} \sum_{u\in \T; \ |u|=n}\ind{U_u\leq x}=x,\quad \ProbS\textit{-almost surely},$$
and similarly
$$\lim_{n\to\infty}\frac{1}{Z_n^j} \sum_{|u|=n,\type(u)=j}\ind{U_u\leq x}=x,\quad \ProbS\textit{-almost surely}.$$
For the deterministic characteristic $\ind{k\geq 1}$ and the corresponding branching process counted with this characteristic, by applying  \cite[Proposition 4.1 {\it i)}]{kolesko-sava-huss-char} we get
\begin{align*}
\frac{Z_n^{\Phi_x^{\mathsf{t}}}}{\rho^{n}}&=\frac{1}{\rho}\cdot\overbrace{\frac{1}{\rho^{n-1}}\sum_{k\geq 0}^{n-1}\sum_{|u|=k}1}^{\AS W\sum_{k\geq 0}\rho^{-k}}+\overbrace{\frac{1}{\rho^{n}}\sum_{|u|=n}\ind{U_u\leq x}}^{\AS xW}\\
	&\AS \Big(\frac{1}{\rho}\sum_{k\geq 0}\frac{1}{\rho^{k}}+x\Big)W=\Big(\frac{1}{\rho-1}+x\Big)W,\quad \textit{  as }n\to\infty,
	\end{align*}
and this shows the first part of the claim. For the second one, from Kesten-Stigum theorem we know that $\frac{Z_n^j}{\rho^n}\AS W\mathrm{u}_j$ as $n\to\infty$ for any $j\in[J]$, and for the characteristic $\Phi_x^j$,
since it holds
$$Z_n^{\Phi_x^j}=\sum_{k=0}^{n}Z_k^j+\sum_{u\in \T; \ |u|=n+1,\type(u)=j}\ind{U_u\leq x},$$
we obtain
\begin{align*}
\frac{Z_n^{\Phi_x^j}}{\rho^{n}}&=\overbrace{\frac{1}{\rho^{n}}\sum_{k=1}^{n}Z_k^j}^{\AS W\sum_{k\geq 0}\rho^{-k}\mathrm{u}_j}+\overbrace{\rho\frac{1}{\rho^{n+1}}\sum_{|u|=n+1,\type(u)=j}\ind{U_u\leq x}}^{\AS \rho xW \mathrm{u}_j}\\
	&\AS \Big(\frac{1}{\rho}\sum_{k\geq 0}\frac{1}{\rho^{k}}+x\Big)W\rho \mathrm{u}_j=\Big(\frac{1}{\rho-1}+x\Big)W\rho \mathrm{u}_j,\quad \textit{ as }n\to\infty,
	\end{align*}
	and the proof is completed.
\end{proof}

Following the notation from \cite{kolesko-sava-huss-char}, for every characteristic  $\Phi:\Z\to\R^{1\times J}$ we define 
two  vectors $\mathrm{x}_i(\Phi)=\sum_{k\in \N}\E[\Phi(k)]\pi^{(i)}A_i^{-k}$ for $i=1,2$.
In particular, since $\E[\Phi_x^{\mathsf{t}}(0)]=x\mathrm{1}$, we have 
\begin{align}\label{eq:vect-x2}
	\nonumber
\mathrm{x}_i\left(\Phi_x^{\mathsf{t}}\right)&=x\mathrm{1}\pi^{(i)}+\mathrm{1}\pi^{(i)}\sum_{k=1}^{\infty}A_i^{-k}=\mathrm{1}\pi^{(i)}\left(xI+\sum_{k=1}^{\infty}A_i^{-k}\right)\\
&=\mathrm{1}\pi^{(i)}\left(x + (A_i-I)^{-1}\right),\quad \text{for }i=1,2.
\end{align}
\begin{comment}
\textcolor{blue}{We will remove the blue part below in this general form, once we know if and in which form we need them!
For  $l=0,\dots,J$ constants
\begin{align}
\label{eq:definition of sigma_l}
\sigma_l^2(\Phi)\defeq \frac{\rho^{-l}}{(2l+1)(l!)^2}\sum_{\lambda\in\sigma^2_A}\Var \left[\mathrm{x} _2(\Phi) \pi_\lambda (A-\lambda I)^{l}L\right]\mathrm{u}.
\end{align}
Finally
\begin{align}
\label{eq:definition of sigma}
\sigma^2(\Phi)\defeq\sum_{k\in \Z}\rho^{-k}\Var\left[\Phi(k)+\Psi^\Phi(k)\right]\mathrm{u},
\end{align} 
where $\Psi^\Phi$ is the centered characteristic defined as
\begin{align*}
\Psi^\Phi(k)=
\sum_{l\in \Z}\E\Phi(k-l-1)A^l\mathsf{P}(k,l)(L-A),
\end{align*}
with matrices $\mathsf{P}(k,l)$:
\begin{align*}
\mathsf{P}(k,l)\defeq\begin{cases}
-\pi^{(1)}\ind{l<0}+\pi^{(2)}\ind{l\geq 0}+\pi^{(3)}\ind{l\geq 0},\quad \text{if }k\leq 0\\
-\pi^{(1)}\ind{l<0}-\pi^{(2)}\ind{l<0}+\pi^{(3)}\ind{l\geq 0},\quad \text{if }k>0.
\end{cases}
\end{align*}}
\end{comment}
For any random characteristic $\Phi:\Z\to\R^{1\times J}$ that satisfies the assumptions of  \cite[Theorem 3.5]{kolesko-sava-huss-char}, i.e.~(GW1)-(GW3) hold,   $\sum_{k\in \Z}\big\|\E[\Phi(k)]\big\|(\rho^{-k}+\vartheta^{-k})<\infty$ for some $\vartheta<\sqrt{\rho}$, and finally $\sum_{k\in \Z}\|\Var[\Phi(k)]\|\rho^{-k}<\infty$, we set 
\begin{equation}\label{eq:fluctuations-phi}
F^\Phi_n=\mathrm{x}_1(\Phi)A^n_1 W^{(1)}+\mathrm{x}_2(\Phi)A^n_2Z_0.
\end{equation}
Recall that the constants $\sigma_{\ell}^2$, for $\ell=0,\ldots, J-1$ are defined as
\begin{equation}\label{eq:sigma_l}
\sigma_{\ell}^2=\sigma_{\ell}^2(\Phi)= \frac{\rho^{-\ell}}{(2\ell+1)(\ell !)^2}\sum_{\lambda\in\sigma^2_A}\Var \Big[\mathrm{x} _2(\Phi) \pi_\lambda (A-\lambda I)^{\ell}L\Big]\mathrm{u}.
\end{equation}
Let $\ell$ be the maximal integer such that $\sigma_{\ell}(\Phi)>0$ and we set $\ell=-\frac12$ if there is no such integer. Then Theorem 3.5 of  \cite{kolesko-sava-huss-char} states that, for a standard normal variable $\mathcal{N}(0,1)$ independent of $W$, the following stable convergence holds
\begin{align}
\label{eq:clt}
\frac{Z^\Phi_n-F^\Phi_n}{n^{\ell+\frac12}\rho^{n/2}\sqrt{W}}\stabSto G^{\Phi}, \quad \text{as } n\to\infty,
\end{align} 
where $G^\Phi= \sigma_{\ell}(\Phi) \mathcal{N}(0,1)$ if not all $\sigma_{\ell}$ are zero, and  $G^\Phi= \sigma(\Phi)\mathcal{N}(0,1)$ otherwise, while $\sigma(\Phi)$  is  defined by
\begin{align}
	\label{eq:definition of sigma}
	\sigma^2(\Phi)=\sum_{k\in \Z}\rho^{-k}\Var\left[\Phi(k)+\Psi^\Phi(k)\right]\mathrm{u},
\end{align} 
and $\Psi^\Phi$ is the centered characteristic given by
\begin{align*}
	\Psi^\Phi(k)=
	\sum_{\ell\in \Z}\E\Phi(k-\ell-1)A^{\ell}\mathsf{P}(k,\ell)(L-A).
\end{align*}
Above, the matrices $\mathsf{P}(k, \ell)$, for $k,\ell\in\Z$ are defined as
\begin{align*}
	\mathsf{P}(k,\ell)=\begin{cases}
		-\pi^{(1)}\ind{\ell<0}+\pi^{(2)}\ind{\ell\geq 0}+\pi^{(3)}\ind{\ell\geq 0},\quad \text{if }k\leq 0\\
		-\pi^{(1)}\ind{\ell<0}-\pi^{(2)}\ind{\ell<0}+\pi^{(3)}\ind{\ell\geq 0},\quad \text{if }k>0.
	\end{cases}
\end{align*}

%
%The application of Theorem 3.5 of  \cite{kolesko-sava-huss-char} yields 
%\begin{theorem}[Theorem 3.5 of  \cite{kolesko-sava-huss-char}]
%	\label{thm:clt for CMJ}
%	Assume conditions $(GW1)-(GW3)$ hold and let
%	$\Phi:\Z\to\C^{1\times J}$ be a  random characteristic of dimension $d=1$ that satisfies \eqref{eq:main-assum1} and \eqref{eq:main-assum2}. Then for a standard normal random variable $\mathcal{N}$ the following dichotomy holds.
%	\begin{enumerate}[i)]
%		\item If $\sigma_l=0$ for all $0\le l\leq J-1$,  then there exists a constant $\sigma\ge0$ such that either $\sigma>0$ and
%		\begin{align*}
%		\frac{Z_n^{\Phi}-\mathrm{x}_1A^n_1 W^{(1)}-\mathrm{x}_2A^n_2Z_0 }{\rho^{n/2}}\stabto \sigma\sqrt W\cdot G,\quad \text{as } n\to\infty
%		\end{align*} 
%		or $\sigma=0$ and the left hand side above is a deterministic sequence converging to 0.
%		\item Otherwise, let $0\le l\le J-1$ be maximal such that $\sigma_l\neq0$. Then
%		\begin{align*}
%		\frac{Z_n^{\Phi}-\mathrm{x}_1A_1^n W^{(1)}-\mathrm{x}_2A_2^nZ_0 }{n^{l+\frac 12}\rho^{n/2}}\stabto \sigma_l\sqrt W\cdot G, \quad \text{as } n\to\infty.
%		\end{align*}
%		In both cases above the limiting normal variable $\mathcal{N}$ is independent of $W$.
%	\end{enumerate}
%\end{theorem}
%
%
%In this note we give yet another applications of our model of the discrete multitype branching processes with a characteristic.  Before, we will also need estimates for the higher moments of $(Z_n^{\Phi})_{n\in\mathbb{N}}$.

\section{The embedding of the urn model into the branching process}\label{sec:embed}

\textbf{Notation.}
We slightly abuse the notation and write $\Phi^{\mathsf{t}}$ (respectively $\Phi^j$, $j\in[J]$) for the whole family $\big\{\Phi^{\mathsf{t}}_x\big\}_{x\in[0,1)}$ (respectively $\big\{\Phi^j_x\big\}_{x\in[0,1)}$) of characteristics indexed over the threshold $x\in[0,1)$. We denote by $\mathcal C$ the set of characteristics $\Phi$ which are linear combinations of   $\Phi^{\mathsf{t}}$ and $\Phi^j$, for $j\in[J]$. Again by abuse of notation, by $\Phi\in\mathcal{C}$ we actually refer to the whole family of characteristics $\{\Phi_x\}_{x\in[0,1)}$, that is:
$$\Phi=\big\{\Phi_x\in\mathcal{C}:\ x\in[0,1)\big\}=\big\{a \Phi^{\mathsf{t}}_x+b \Phi^j_x:\ a,b\in\R,\ j\in[J], \ x\in [0,1)\big\}.$$

\textbf{Extension of $x\in[0,1)$ to $x\in\R$.}  Instead of working with thresholds $x\in[0,1)$ and corresponding characteristics $\Phi_x$, we can extend the domain of $x$ to the whole $\mathbb{R}$ as following.
For any $\Phi\in\mathcal C$ and any $x\in\R$ we define 
$$\Phi_{x}(k)=\Phi_{\{x\}}(k+\floor{x}).$$
The corresponding CMJ process $(Z_n^{\Phi_x})_{n\in\mathbb{N}}$ satisfies $Z_n^{\Phi_x}=Z_0^{\Phi_{x+n}}$ for every $n\in\N$ and $x\in\R$ and we finally define
$$\mathcal{Z}^{\Phi}(x)= Z_0^{\Phi_{x}},$$ 
and similarly $\Fphi(x)= F_0^{\Phi_{x}}$. For any $x\in\R$, \eqref{eq:clt} yields the existence of a Gaussian process $\{\Gphi(x);\ x\in\R\}$ such that the following convergence holds
\begin{align}\label{eq:fin-dim}
	\frac{\mathcal Z^{\Phi}(x+n)-\Fphi(n+x)}{n^{\ell+\frac12}\rho^{n/2}\sqrt{W}}\stabSto \Gphi(x), \quad \text{as } n\to\infty.
\end{align}
Cram\'er-Wold device implies that, in fact, the convergence holds for finite dimensional distributions and the limiting process $\Gphi$ is jointly Gaussian with  $\Gphi(x)\eqdist \rho^{\floor{x}/2}\sigma_\ell(\Phi_{\{x\}}) \mathcal N(0,1)$ or $\Gphi(x)\eqdist \rho^{\floor{x}/2}\sigma(\Phi_{\{x\}}) \mathcal N(0,1)$  depending on the value of the constants $\sigma_\ell$, respectively $\sigma$.
Furthermore, we  write $\Zt, \Ft,\Gt$ (respectively $\Zj,\Fj,\Gj$) for  $\Zphi, \Fphi,\Gphi$ if $\Phi=\Phi^{\mathsf{t}}$ (respectively $\Phi=\Phi^j$). 
%For the rest of the paper, for any $x\in[0,\infty)$, any $j\in[J]$ the following notation is used:
%$$\Zt(x):= Z_{\lfloor x\rfloor}^{\Phi^{\mathsf{t}}_{\{x\}}}\quad \text{and} \quad \Zj(x):= Z_{\lfloor x\rfloor}^{\Phi^{j}_{\{x\}}}.$$ 
%In particular, for the  linear combination $\Phi={\rho}\mathrm u_j\Phi^{\mathsf{t}}-\Phi^j$ of $\Phi^{\mathsf{t}}$ and $\Phi^j$, in view of the linearity of the branching process with characteristic $\Phi$, we can write
%$$\mathcal{Z}^{\Phi}(x)=\rho\mathrm{u}_j\Zt(x)-\Zj(x).$$
Since with probability one, all the random variables $(U_u)_{u\in \U}$ are different, the process $(\Zt(x))_{x\geq 0}$ at its jump point increases by 1. Therefore, the following stopping times are well-defined: for $k\in\N$, define $\tau_k$ as
\begin{equation}\label{eq:tau-urn}
\tau_k=\inf\big\{x\geq 0:\ \Zt(x)=k\big\}.
\end{equation}
\begin{remark} With the stopping times $(\tau_k)_{k\in\N}$ just introduced, we have
$$B_j(k)=\Zj(\tau_k),$$
and this is exactly the number of balls of type $j$ added to the two urns after $k$  draws, for which
we seek first and second order asymptotics. 
Our strategy is as follows: we first prove  functional limit theorems for the processes $\{\Zt(x);\ x\in \R\}$  and $\{\Zj(x);\ x\in\R\}$, and then we conclude the corresponding limit theorems for $B_j(k)$, for $j\in[J]$.
We start with the description of the leading term in the asymptotics of $\mathcal{Z}^{\Phi}(x)$, for any characteristic $\Phi\in \mathcal{C}$. \end{remark}

\textbf{Periodic functions.} For any $\lambda\in \C$, we introduce the function 
\begin{equation}\label{eq:lper}
l_\lambda:[0,\infty)\to\R\quad \text{defined as}\quad l_\lambda(x)= (1+(\lambda -1)\{x\})\lambda^{-\{x\}},
\end{equation}
where  $\lambda^t=e^{t\log \lambda}$ and $z\mapsto \log z$ is the principal branch of the logarithm. The function $l_\lambda$ is continuous and 1-periodic and it satisfies
$$\lambda^xl_{\lambda}(x)=\lambda^{\lfloor x\rfloor}(1+(\lambda-1)\{x\}).$$
 Moreover, the mapping $x\mapsto \lambda^xl_{\lambda}(x)$ equals $\lambda^x$ for integer $x$ and is linear in between.
\begin{proposition}\label{lln}
Assume (GW1),(GW2), and (GW4) hold, and for any $j\in[J]$ let $\Phi=a\Phi^{\mathsf{t}}+b\Phi^j$, with $a,b\in\R$. Then it holds
	\begin{align}
		\label{eq:first order in urn model}
		\lim_{x\to\infty}\frac{\mathcal{Z}^\Phi(x)}{\rho^{x}l_{\rho}(x)}=\lim_{x\to\infty}\frac{\mathcal{Z}^\Phi(x)}{\rho^{\lfloor x\rfloor}(1+(\rho-1)\{x\})}=(a+b\rho\mathrm{u}_j)\frac{1}{\rho-1}W,\quad \ProbS\text{-almost surely}.
	\end{align}
%	where the function $f:[0,\infty)\to\mathbb{R}$ is defined as $f(x)=\log_{\rho}(1+(\rho-1)x)-x$.
\end{proposition}
\begin{proof}
Because of the linearity of  CMJ processes, for $x\in \R$ it holds
$$\mathcal{Z}^{\Phi}(x)=a\Zt(x)+b\Zj (x)\quad \text{ for any } a,b\in\R,\ \text{ and }j\in[J],$$
and it suffices to prove the $\ProbS$-almost sure convergence for $\Zt(x)$ and $\Zj(x)$ separately, as $x\to\infty$. 
% Together with the continuous mapping theorem the claim follows. 
\\
\textbf{Case 1: $x\in[0,1)$.} For $n\in\N$ and $x\in[0,1)$, since $\Zt(x+n)=Z_n^{\Phi^{\mathsf{t}}_x}$, in view of Proposition \ref{prop:lln1}  we get
\begin{equation}\label{eq:as-x01}
\frac{\Zt(x+n)}{\rho^n(1+(\rho-1)x)}\to\frac{1}{\rho-1}W,\quad \ProbS\textit{-almost surely as } n\to\infty.
\end{equation}
\textbf{Case 2: $x\in[0,\infty)$.} This case can be reduced to the previous one where $x\in[0,1)$ as following.
In view of equation \eqref{eq:as-x01}, for any $x\in [0,\infty)$, with $m=n+\lfloor x \rfloor$ we obtain
\begin{align*}
\frac{\Zt(x+n)}{\rho^{x+n}l_\rho(x)}=\frac{\Zt(\lfloor x\rfloor+n+\{x\})}{\rho^{\lfloor x\rfloor+n}(1+(\rho-1)\{x\})}=\frac{\Zt(\{x\}+m)}{\rho^{m}(1+(\rho-1)\{x\})}\AS \frac{1}{\rho-1}W,\quad\textit{ as } m\to \infty,
\end{align*}
where in the last equation above we have used that 
$$\rho^{x+n}l_\rho(x)=\rho^{\lfloor x\rfloor+n}(1+(\rho-1)\{x\})=\rho^{x+n}l_\rho(x+n).$$
We still have to prove that the above $\ProbS$-almost sure convergence holds for $x\to\infty$, that is we have
\begin{align*}
\lim_{x\to\infty}\frac{\Zt(x)}{\rho^{x}l_{\rho}(x)}=\frac{1}{\rho-1}W,\quad \ProbS\text{-almost surely}.
\end{align*}
Indeed, from any sequence tending to infinity we may choose a subsequence $(x_n)$ such that $\{x_n\}$ converges to some $x_0\in\R$. Then for any $\delta>0$ and large $n$, in view of
$$\Zt(\lfloor x_n\rfloor+ x_0-\delta)=Z_{\lfloor\lfloor x_n\rfloor+ x_0-\delta\rfloor}^{\Phi^{\mathsf{t}}_{\{\lfloor x_n\rfloor+x_0-\delta\}}}=Z_{\lfloor\lfloor x_n\rfloor+ x_0-\delta\rfloor}^{\Phi^{\mathsf{t}}_{\{x_0-\delta\}}},$$
we have
\begin{align*}
&\frac{\Zt(\lfloor x_n\rfloor+ x_0-\delta)}{\rho^{\lfloor x_n\rfloor+ x_0-\delta}l_\rho(x_0-\delta)}\cdot\rho^{x_0-\{x_n\}-\delta} \cdot\frac{l_{\rho}(x_0-\delta)}{l_{\rho}(x_n)}
\le
\frac{\Zt(x_n)}{\rho^{x_n}l_{\rho}(x_n)}\\
&
\le
\frac{\Zt(\lfloor x_n\rfloor+ x_0+\delta)}{\rho^{\lfloor x_n\rfloor+ x_0+\delta}l_{\rho}(x_0+\delta)}\rho^{x_0-\{x_n\}+\delta}\cdot\frac{l_{\rho}(x_0+\delta)}{l_{\rho}(x_n)}.
\end{align*}
Taking first the limit as $n$ goes to infinity and then as $\delta$ goes to 0 we get the desired convergence since  $x\mapsto l_{\rho}(x)$ is uniformly continuous.
The same argument can be used in order to show that for any $j\in[J]$ we have 
$$\lim_{x\to\infty}\frac{\Zj(x)}{\rho^{x}l_\rho(x)}=\frac{1}{\rho-1}W\rho\mathrm{u}_j, \quad \ProbS-\text{almost surely},$$
and this proves the claim. 
\end{proof}
An immediate consequence of Proposition \ref{lln} is the following corollary.
\begin{corollary} Under the assumptions of Proposition \ref{lln}, for $\Phi={\rho}\mathrm u_j\Phi^{\mathsf{t}}-\Phi^j$, we have
\begin{align*}
		\lim_{x\to\infty}\frac{\mathcal{Z}^\Phi(x)}{\rho^{x}l_\rho(x)}=\lim_{x\to\infty}\frac{\mathcal{Z}^\Phi(x)}{\rho^{\lfloor x\rfloor}(1+(\rho-1)\{x\})}=0,\quad \ProbS\text{-almost surely}.
	\end{align*}
\end{corollary}
Also, the strong law of large numbers for $(B_j(k))_{k\in\N}$ follows immediately from Proposition \ref{lln}.
\begin{proof}[Proof of Theorem \ref{thm:main_SSLN}]
Since $\tau_k$ goes to infinity as $k$ does, we have
\begin{align*}
\lim_{k\to\infty}\frac{B_j(k)}{k}&=\lim_{k\to\infty}\frac{\Zj(\tau_k)}{\Zt(\tau_k)}=\lim_{k\to\infty}\frac{\Zj(\tau_k)}{\rho^{\lfloor \tau_k\rfloor}(1+(\rho-1)\{\tau_k\})}\cdot \frac{\rho^{\lfloor \tau_k\rfloor}(1+(\rho-1)\{\tau_k\})}{\Zt(\tau_k)}\\
&=\frac{1}{\rho-1}W\rho\mathrm{u}_j\cdot \frac{1}{\frac{1}{\rho-1}W}=\rho\mathrm{u}_j,\quad \ProbS\text{-almost surely}, 
\end{align*}
and this finishes the proof.
\end{proof}

\section{Proof of the main result}\label{sec:main}

The proof is completed in several steps:
\vspace{-0.25cm}
\begin{itemize}\setlength\itemsep{0em}
\item In Lemma \ref{lem:difference of composition of F} we investigate compositions of the fluctuations $\Ft$ and $\Fj$.
\item In Theorem \ref{thm:clt-general} we prove weak convergence of the processes $\mathcal{X}^{\mathsf{t}}=\Zt-\Ft$ and $\mathcal{X}^j=\Zj-\Fj$ (rescaled appropriately) to Gaussian processes $\Gt$ and $\Gj$ respectively, for any $j\in[J]$.
\item Continuity and strict positivity of the variances of the limiting processes $\Gt$ and $\Gj$ are analysed in Proposition \ref{prop:pos-variances} and in Lemma \ref{lem:properties of the gaussian process}.
\item Localization of the stopping times $\tau_n$ is done in Proposition \ref{prop:loc-stop-times}.
\item Finally, the limit theorems for $B_j(n)$ are given in Proposition \ref{lem:clt for fluctuations} and in Theorem \ref{thm:general_limit_theorem}.
\end{itemize}

\subsection{Leading asymptotic terms} 

We start with the description of the leading term in the asymptotics of $\Zt$ and $\Zj$ respectively for any $j\in[J]$. We recall first that for a characteristic $\Phi\in\mathcal{C}$ the  leading term in the asymptotics of $\Zphi$  is given by $\Fphi$ and for simplicity of notation, for $x\in\R$ we write 
$$\mathcal{X}^{\Phi}(x)=\Zphi(x)-\Fphi(x).$$
In particular, for $\Phi=\Phi^{\mathsf{t}}$ and $\Phi=\Phi^j$ respectively, we write 
\begin{align*}
\mathcal{X}^{\mathsf{t}}(x)=\Zt(x)-\Ft(x)\quad \text{and} \quad
\mathcal{X}^j(x)=\Zj(x)-\Fj(x).
\end{align*}

%Since both $\Ft$ and $\Fj$ at arguments $x\in\mathbb{N}$ are linear combinations of exponential functions and between two consecutive integers are linear, we can use the functions $l_{\lambda}$ from \eqref{eq:lper} and write for $x\in [0,\infty)$
%\textcolor{blue}{the two sums below are over the eigenvalues in $\sigma_A^{1}$}
%\begin{align*}
%	\Ft(x)=\sum_{\lambda\in\sigma_A}c_{\lambda}W_{\lambda}\lambda^xl_{\lambda}(x)\quad
%	\text{and}\quad
%	\Fj(x)=\sum_{\lambda\in\sigma_A}c_{\lambda}^jW_{\lambda}\lambda^xl_{\lambda}(x)
%\end{align*}
%for  constants $c_\lambda,c^j_\lambda$ and scalar random variables $W_\lambda$, $\lambda\in \sigma_A$. 
%\textcolor{blue}{How are the random variables $W_{\lambda}$ defined in terms of  $\pi_{\lambda}W^{(1)}$?}
%Since $\Ft$ is eventually increasing almost surely, the inverse function $(\Ft)^{-1}$ is well defined for
%large arguments.
%For simplicity of notation, we write $(\Ft)^{-1}=\Fi$.
\begin{lemma}\label{lem:difference of composition of F}
Assume  (GW1)-(GW3) hold. Then for sufficiently large argument the inverse function $\Fi=(\Ft)^{-1} $ is well defined and for every $j\in[J]$ we have 
	\begin{align}
		\label{eq:composition is asymptotic linear}
		\lim_{t\to\infty}\sup_{s\ge 1}s^{-1}\Big|\Fj\big(\Fi(t+s)\big)-\Fj\big(\Fi(t)\big)-\rho \mathrm{u}_j s\Big| =0,\quad \ProbS\textit{-almost surely}.
	\end{align}
\end{lemma}
\begin{proof}
For any $k\in\N_0,  x\in[0,1)$,  the equality $\E [\Phi_x^{\mathsf{t}}(k)]=(1-x)\E [\Phi_0^{\mathsf{t}}(k)]+x\E[\Phi_0^{\mathsf{t}}(k+1)]$
together with equation \eqref{eq:fluctuations-phi}, gives us
	$$F_{n}^{\Phi_{x}^{\mathsf t}}=(1-x)F_{n}^{\Phi_{0}^{\mathsf t}}+xF_{n+1}^{\Phi_{0}^{\mathsf t}},$$
	that is, for any $x\ge0$, $\Ft(x)$  is just a linear interpolation between  $\Ft(\floor x)$ and $\Ft(\floor x+1)$. 
	On the other hand, as for $n\in\N$ in view of \eqref{eq:vect-x2} and \eqref{eq:fluctuations-phi} we have
	$$\Ft(n)=\rho^n\mathrm{x}_1(\Phi_{0}^{\mathsf t}) A_1^n\pi_\rho W^{(1)}=\tfrac{\rho^{n}}{\rho-1}W+o(\rho^n),$$
	we conclude that the following holds:
	\begin{align*}
		&\Ft(x)=\tfrac{\rho^{x}}{\rho-1}l_\rho(x)W+o(\rho^x)\text{ and }\\
		&\big(\Ft\big)'(x)=\rho^{\floor x}W+o(\rho^x)\text{ for }x\notin\Z.
	\end{align*}
	By the same argument, for $j\in[J]$ we obtain
	\begin{align*}
		&\Fj(x)=\tfrac{\rho^{x}}{\rho-1}l_\rho(x)\rho\mathrm{u}_jW+o(\rho^x)\text{ and }\\
		&\big(\Fj\big)'(x)=\rho^{\floor x}\rho\mathrm{u}_jW+o(\rho^x)\text{ for }x\notin\Z.
	\end{align*}
In particular, $\Ft$ is eventually increasing $\ProbS$-almost surely and thus  the inverse function $(\Ft)^{-1}$ is well-defined for large arguments.
Furthermore, if $\Fi(t)\notin\N$ and $t$ is large enough then we have
$$\big(\Fj\circ \Fi\big)'(t)=(\Fj)'\big(\Fi(t)\big) \cdot (\Fi(t))'=\frac{(\Fj)'\big(\Fi(t)\big)}{(\Ft)'\big(\Fi(t)\big)}\to\rho\mathrm{u}_j,\quad \text{as }t\to\infty,$$
	since $\Fi(t)$ diverges to infinity.
Finally, as for large $t$ it holds
$$
\Fj\big(\Fi(t+s)\big)-\Fj\big(\Fi(t)\big)=\int_t^{t+s}\big(\Fj\circ \Fi\big)'(u)du
$$	
we obtain \eqref{eq:composition is asymptotic linear}.
\end{proof}

\subsection{Limit theorems for $\Xt$ and $\Xj$}

In order to prove weak convergence of the processes $\Big\{\frac{1}{n^{\ell+\frac12}\rho^{n/2}\sqrt W}\Xj(n+x); \ x\in \R\Big\}$ we follow the well-known technique: we first prove weak convergence
of the finite-dimensional distributions and then we prove tightness.
According to \cite[Theorem 3.5]{kolesko-sava-huss-char} 
the finite dimensional distributions of the aforementioned processes converge jointly; see also equation \eqref{eq:fin-dim} and the discussion after.

\begin{theorem}\label{thm:clt-general} Suppose that (GW1)-(GW3) hold and  $L$ satisfies $\E\left[\|L\|^{2+\delta}\right]<\infty$ for some $\delta\in (0,1)$. Then for every $j\in [J]$, the family of distributions 
$$\bigg\{\frac{1}{n^{\ell+\frac12}\rho^{n/2}}\Xj(n+x); \ x\in \R\bigg\}$$
with respect to $\Prob$ is tight in the Skorokhod space $\mathcal D(\R)$ endowed with the standard $\mathbf{J}_1$  topology.
\end{theorem}
\begin{proof}
First, let us observe that for any $k\in\Z,\ m\in\N$ the concatenation is a continuous mapping from $\mathcal{D}([k,k+1))\times\dots\times\mathcal D([k+m-1,k+m))$ to $\mathcal D([k,k+m))$. Consequently, it suffices to prove tightness in the space $\mathcal D([0,1))$. For $x\in[0,1)$, $j\in[J]$, and $n\in\N$ we set
$$\Yj(n+x)=\frac{1}{n^{\ell+\frac12}\rho^{n/2}}\Xj(n+x),$$
where $\ell$ may be $-\frac12$ in case {\it i)} of \cite[Theorem 3.5]{kolesko-sava-huss-char}, and the characteristic $\Phi^j=(\Phi^j_x)_{x\in[0,1)}$ is defined as in \eqref{eq:char-urn-type}.
In view of  \cite[Theorem 13.5]{billingsley_2nd_edition} it suffices to show that  for any $ 0\leq x\le y\le z<1$, $\lambda>0$ and $n$ large enough it holds
\begin{align*}
\Prob\Big(\big|\Yj(n+y)-\Yj(n+x)\big|\wedge\big|\Yj(n+z)-\Yj(n+y)\big|\ge \lambda\Big)\le C\lambda^{-2p}|z-x|^{p},
 \end{align*}
where $p\defeq(2+\delta)/2\in (1,3/2)$ and $C>0$ is some constant. A more restrictive condition is the following inequality
\begin{align*}
\E\Big[\left |\Yj(n+y)-\Yj(n+x)\right|^p\cdot\left |\Yj(n+z)-\Yj(n+y))\right|^p\Big]\le C|z-x|^{p}.
\end{align*}
By Hölder's inequality, we have
\begin{align*}
\E&\Big[|\Yj(n+y)-\Yj(n+x)|^{p}\cdot |\Yj(n+z)-\Yj(n+y)|^{p}\Big]\\
%	&\leq\E\Big[|\Yj(n+y)-\Yj(n+x)|^{2}\cdot  |\Yj(n+z)-\Yj(n+y)|^{2}\Big]^{p/2}\\
	&\le\E\left[\E\Big[|\Yj(n+y)-\Yj(n+x)|^{2}\cdot |\Yj(n+z)-\Yj(n+y)|^{2}\Big|\A_n\Big]^{p/2}\right]
	\end{align*}
so it remains to show the following estimate 
$$\E\Big[|\Yj(n+y)-\Yj(n+x)|^{2}\cdot |\Yj(n+z)-\Yj(n+y)|^{2}\Big|\A_n\Big]\leq H_n|z-x|^2,$$
for some random variables $H_n$ with bounded $p/2$ moment.
Recalling that for $0\leq x<1$, we have
\begin{align*}
\Zj(n+x)=\sum_{k=0}^{n-1}Z_k^j+\sum_{u\in\T;  |u|=n+1; \type(u)=j}\ind{U_u\leq x},
\end{align*}
yields for $0\leq x\leq y<1$  that it holds
\begin{align*}
\Zj(n+y)-\Zj(n+x)=\sum_{u\in \T;|u|=n+1,\type(u)=j}\ind{x< U_u\leq y}.
\end{align*}
We also have
\begin{align*}
\Fj(n+y)-\Fj(n+x)&=\left(\mathrm{x}_1(\Phi^j_y)-\mathrm{x}_1(\Phi^j_x)\right)A^n_1 W^{(1)}+\left(\mathrm{x}_2(\Phi^j_y)-\mathrm{x}_2(\Phi^j_x)\right)A^n_2Z_0\\
&=(y-x)A\mathrm{e}_j^\top\pi^{(1)}A^n_1 W^{(1)}+(y-x)A\mathrm{e}_j^\top\pi^{(2)}A^n_2Z_0\\
&=(y-x)A\mathrm{e}_j^\top\big(\pi^{(1)}A^n_1 W^{(1)}+\pi^{(2)}A^n_2Z_0\big)\\
&=(y-x)F_n^{\Psi}=(y-x)(Z^j_{n+1}-(Z_n^{\Psi}-F_n^{\Psi})),
\end{align*}
where $\Psi:\N_0\to\R^{1\times J}$ is the characteristic given by $\Psi(k)=\ind{k=0}\mathrm{e}_j^{\top}L=\ind{k=0}\langle\mathrm{e}_j,L\rangle$. Hence
$Z_n^{\Psi}$ counts the number of individuals of type $j$ in the $(n+1)$-th generation. %which we then  recover the %multitype Galton-Watson process $(Z_n)_{n\in\N}$, $Z_n^{\Psi}=|Z_n|$ and therefore $\E[Z_n^{\Psi}]=o(\rho^{-n})$. We also have
%$$
%\frac{X^\Psi_n}{n^{l+\frac12}\rho^{n/2}\sqrt{W}}=\frac{Z^\Psi_n-F^\Psi_n}{n^{l+\frac12}\rho^{n/2}\sqrt{W}}\stabto G^{\Psi}$$ 
%where $G^\Psi\defeq \sigma_l(\Psi)\mathcal{N}$.
We then obtain 
\begin{align*}
\Yj(n+y)-& \Yj(n+x)=\frac{1}{n^{\ell+\frac12}\rho^{n/2}}\left(\Xj(n+y)-\Xj(n+x)\right)\\
&=\frac{1}{n^{\ell+\frac12}\rho^{n/2}}\left[(\Zj(n+y)-\Zj(n+x))-(\Fj(n+y)-\Fj(n+x))\right]\\
&=\frac{1}{n^{\ell+\frac12}\rho^{n/2}}\Big[\sum_{u\in\T;  |u|=n+1; \type(u)=j}\ind{x< U_u\leq y}-(y-x)Z^j_{n+1}+(y-x)(Z_n^{\Psi}-F_n^{\Psi})\Big]\\
&=\frac{1}{n^{\ell+\frac12}\rho^{n/2}}\Big[\sum_{u\in\T;  |u|=n+1; \type(u)=j}\left(\ind{x< U_u\leq y}-(y-x)\right)\cdot 1+(y-x)(Z_n^{\Psi}-F_n^{\Psi})\Big].
\end{align*}
Since $Z_{n+1}^j$ and $Z_n^{\Psi}-F_n^{\Psi}$ are $\A_n$-measurable, by applying   Lemma \ref{lem:auxilarity lemma} to
the intervals $I=(x,y)$, $J=(y,z)$, and $a_i=1$, $A=(y-x)(Z_n^{\Psi}-F_n^{\Psi})$  and $B=(z-y)(Z_n^{\Psi}-F_n^{\Psi})$, we finally obtain
\begin{align*}
\E & \Big[|\Yj(n+y)-\Yj(n+x)|^{2}\cdot |\Yj(n+z)-\Yj(n+y)|^{2}|\A_n\Big]\\
&\le C \Big(\frac{1}{n^{2\ell+1}\rho^{n}}\Big)^2\big((y-x)(z-y)\big(Z_{n+1}^j\big)^2 +A^4+B^4\big)
\\
%&=C\Big(\frac{1}{n^{2\ell+1}\rho^{n}}\Big)^2(y-x)(z-y)Z_{n+1}^j\left[((y-x)^2+(z-y)^2)(Z_n^{\Psi}-F_n^{\Psi})^2+Z_{n+1}^j\right]\\
%&+\Big(\frac{1}{n^{2\ell+1}\rho^{n}}\Big)^2 (y-x)^2\cdot(z-y)^2(Z_n^{\Psi}-F_n^{\Psi})^4\\
& \leq C (z-x)^2 \Big(\frac{1}{n^{\ell+\frac12}\rho^{n/2}}\Big)^4\left[(Z_n^{\Psi}-F_n^{\Psi})^4+(Z_{n+1}^j)^2\right]=:C|z-x|^2H_n,
\end{align*}
for some absolute constant $C$.
We have $Z_{n+1}^j=Z_n^{\Psi}$, so Theorem \ref{lem:moments growth}{\it i)} implies that  $\E[(Z_{n+1}^j)^p]=O(\rho^{pn})$. On the other hand, from Corollary \ref{cor:moment estimate} we have $\E\big[(Z_n^{\Psi}-F_n^{\Psi})^{2p}\big] =O(n^{(2\ell+1)p}\rho^{pn})$.
%Raising everything to the power $p/2$ and  taking expectation, we obtain
%\begin{align*}
%\E&\Big[|\Yj(n+y)-\Yj(n+x)|^{p}\cdot |\Yj(n+z)-\Yj(n+y)|^{p}\Big]\\
%&\leq C \Big(\frac{1}{n^{2\ell+1}\rho^n }\Big)^{p}|z-x|^p\E\left[(Z_n^{\Psi}-F_n^{\Psi})^{2p}\right]\leq C |z-x|^p,
%\end{align*}
%where in the last inequality above we have used that $L$ has $2+\delta=2p$ finite moments, and we have applied Corollary \ref{cor:moment estimate} to the characteristic $\Psi$, which gives that $\E\left[(Z_n^{\Psi}-F_n^{\Psi})^{2p}\right]=O\left(n^{(2\ell+1)p}\rho^{np}\right)$.
% Therefore
%$$\E\Big[\left |\Yj(n+y)-\Yj(n+x)\right|^p\cdot\left |\Yj(n+z)-\Yj(n+y))\right|^p\Big]\le C|z-x|^{p}$$
As a consequence the random variables  $H_n$ have bounded $p/2$ moments and in turn the process $\Big\{\frac{1}{n^{\ell+\frac12}\rho^{n/2}}\Xj(n+x); \ x\in \R\Big\}$
is tight in  $\mathcal D(\R)$.
\end{proof}
%The proof above works if we replace  $\Psi$ with the constant characteristic $\ind{k=0}\mathrm{1}$, and in this case we obtain the tightness of the family of distributions $\bigg\{\frac{1}{n^{l+\frac12}\rho^{n/2}\sqrt W}\Xt(n+x); \ x\in [0,1)\bigg\}$. 
As a consequence of the previous result we obtain the following:
\begin{corollary}
Suppose that (GW1)-(GW3) hold and the matrix $L$ satisfies $\E\left[\|L\|^{2+\delta}\right]<\infty$ for some $\delta\in (0,1)$. Then the family of distributions
$$\Big\{\frac{1}{n^{\ell+\frac12}\rho^{n/2}}\Xt(n+x); \ x\in \R\Big\}$$
is tight in the Skorokhod space $\mathcal D(\R)$ endowed with the standard $\mathbf{J}_1$  topology.
\end{corollary}
\begin{proof}
In view of $\Zt(n+x)-1=\sum_{j=1}^J \Zj(n-1+x)$ and  $\Ft(n+x)=\sum_{j=1}^J \Fj(n-1+x)$, together with the equality
\begin{align*}
\Yt(n+x)=\frac{1}{n^{\ell+\frac12}\rho^{n/2}}\Xt(n+x)=\sum_{j=1}^J\frac{1}{n^{\ell+\frac12}\rho^{n/2}}\Xj(n-1+x),
\end{align*}
we see that $\Yt(n+x)$ can be written as a finite sum of tight processes, so it is tight as well.
\end{proof}

%Since we have the weak convergence
%of the finite-dimensional distributions and the tightness of the involved processes, this implies the weak convergence to Gaussian processes   $\Gphi(x)\defeq \rho^{\lfloor x\rfloor/2}G^{\Phi_{\{x\}}}$, where we once again remind that $G^{\Phi_{\{x\}}}=\sigma_l(\Phi_{\{x\}})\mathcal{N}(0,1)$ or $G^{\Phi_{\{x\}}}=\sigma(\Phi_{\{x\}})\mathcal{N}(0,1)$, where $\mathcal{N}(0,1)$ is a standard normal variable independent of $W$. Moreover, $\Gt$ (respectively  $\Gphi$) stays for $\Gphi$ if $\Phi=\Phi^{\mathsf{t}}$ (respectively $\Phi=\Phi^i$). 

The convergence of the finite dimensional distributions together with the tightness gives the weak convergence. 
\begin{proposition}\label{thm:conv-to-gauss}
Suppose that (GW1)-(GW3) hold and the matrix $L$ satisfies $\E\left[\|L\|^{2+\delta}\right]<\infty$ for some $\delta\in (0,1)$. Then  we have the following weak convergence of sequences of processes in  the Skorokhod space $\mathcal D(\R)$ endowed with the standard $\mathbf{J}_1$  topology: for every $j\in[J]$ it holds
$$\bigg\{\frac{1}{n^{\ell+\frac12}\rho^{n/2}\sqrt W}(\Xt(n+x),\Xj(n+x)); \ x\in \R\bigg\}\stabSto \{(\Gt(x),\Gj(x));\ x\in\R\}.$$
\end{proposition}

\subsection{Properties of the limiting processes $\Gt$ and $\Gj$}

Remark that for $\Phi_0\in\mathcal{C}$, we have
$$\Phi_0(1)=a \Phi_0^{\mathsf{t}}(1)+b\Phi_0^j(1)=a\mathrm{1}+b\mathrm{e}_j^\top L \quad \text{and} \quad \E[\Phi_0(1)]=a\mathrm{1}+b\mathrm{e}_j^\top A,$$
where $a,b\in \R$ and $j\in[J]$. On the other hand, taking $a=-\rho\mathrm{u}_j$
and $b=1$, we recover
$$\mathrm{w}_j=\mathrm{e}_j^{\top}A-\rho\mathrm{u}_j\mathrm{1}=\E[\Phi_0^j(1)-\rho\mathrm{u}_j\Phi_0^{\mathsf{t}}(1)]$$
as defined in \eqref{eq:w_j} whose $i$-th entry is given by 
$\mathrm{w}_{ji}=\E[L^{(i,j)}-\rho\mathrm{u}_j]=a_{ji}-\rho\mathrm{u}_j$.
\begin{proposition}\label{prop:pos-variances}
For any $\Phi\in\mathcal C$ and $j\in[J]$ assume that $\mathrm{w}_j\ne \mathrm{0}$ and that 
	\vspace{-0.25cm}
	\begin{align}
		\label{eq:condition for nondegenerate sigma}
		\sum_{\lambda\in\sigma_A} \sum_{\ell=0}^{J-1}\|\Var(\mathrm{w}_jN^\ell\pi_{\lambda}L)\|>0.
	\end{align}
\begin{enumerate}[i)]
\setlength\itemsep{0em}
\item If 
$\sum_{\lambda\in\sigma_A^{2}}\sum_{\ell=0}^{J-1}\|\Var(\mathrm{w}_jN^\ell\pi_{\lambda}L)\|>0$,
then for any $x\in[0,1)$ it holds
\begin{align}
\label{eq:definition of l}
\max\{0\leq \ell\leq J-1:\sigma_\ell^2(\Phi_x)>0\}=\max\Big\{\ell\ge 0:\sum_{\lambda\in\sigma^{2}_A}\|\Var(\mathrm{w}_jN^\ell\pi_{\lambda}L)\|>0\Big\}.
		\end{align}
		In particular, the largest $\ell$ such that $\sigma_\ell^2(\Phi_x)>0$ does not depend on $x$.
		\item Otherwise, for any $x\in[0,1)$ and $0\leq \ell\leq J-1$ it holds
		$$\sigma_\ell^2(\Phi_x)=0\text{ and }\sigma^2(\Phi_x)>0.$$ 
\end{enumerate}
\end{proposition}
\begin{proof}
	\textit{i)} For any $x\in[0,1)$, the vector $\mathrm{x}_2(\Phi_x)$ is given by
	\begin{align*}
		\mathrm{x}_2(\Phi_x)=\sum_{k=0}^{\infty}\E[\Phi_x(k)]\pi^{(2)}A_2^{-k}=x\mathrm{w}_j\pi^{(2)}+\mathrm{w}_j\pi^{(2)}\sum_{k=1}^{\infty}A_2^{-k}.
	\end{align*}
	If $k$ is  at least the right hand side of equation \eqref{eq:definition of l} then for any $\lambda\in\sigma^{2}_A$ we have, almost surely $\mathrm{w}_jN_\lambda^k(L-A)= 0$. Since $A_2$ is invertible, we have 
	$$\sum_{j\geq 1}A_2^{-j}=A_2^{-1}\sum_{j\geq 0}A_2^{-j}=\sum_{j\geq 0}A_2^{-j}-I,$$
	so from the last two matrix equations we get $\sum_{j\geq 0}A_2^{-j}=A_2\sum_{j\geq 0}A_2^{-j}-A_2$, which in turn implies
	$(A_2-I)\sum_{j\geq 0}A_2^{-j}=A_2$. Multiplying this equation with $(A_2-I)^{-1}$ from the right and with $A_2^{-1}$ from the left, we obtain $\sum_{j\geq 1}A_2^{-j}=(A_2-I)^{-1}$ so $(A_2-I)\sum_{j\geq 1}A_2^{-j}\pi_{\lambda}=\pi_{\lambda}$, and finally it follows that $\sum_{j\geq 1}A_2^{-j}\pi_{\lambda}=((\lambda-1)I+N_{\lambda})^{-1}$.	
In view of the definition \eqref{eq:sigma_l} of  $\sigma_\ell(\Phi)$, it is enough to understand $\mathrm{x} _2(\Phi) \pi_\lambda (A-\lambda I)^{\ell}(L-A)$.	
It holds
	\begin{align*}
		\mathrm{x} _2(\Phi_x) \pi_\lambda (A-\lambda I)^{k}(L-A)&=\Big(x\mathrm{w}_j\pi_\lambda+\mathrm{w}_j\sum_{j=1}^{\infty}A_2^{-j}\pi_\lambda\Big) (A-\lambda I)^{k}\pi_\lambda(L-A)\\
		&=\big(x\mathrm{w}_j\pi_\lambda+\mathrm{w}_j((\lambda-1)I +N_\lambda)^{-1}\big) N_\lambda^{k}(L-A)\\
		&=\bigg(x\mathrm{w}_j\pi_\lambda+\mathrm{w}_j\sum_{j=0}^{d_\lambda-1}{-1\choose j}(\lambda-1)^{-1-j}N_\lambda^j\bigg) N_\lambda^{k}\pi_\lambda(L-A)= 0,
	\end{align*}
	almost surely, and hence $\sigma_\ell^2(\Phi_x)=0$. On the other hand, if 
	$0\leq \ell\leq J-1$ is the  maximal number such that $\Prob(\mathrm{w}_jN_\lambda^\ell(L-A)\neq 0)>0$ for some $\lambda\in\sigma^{2}_A$, then for such $\ell$ and $\lambda$, similar calculations give
	\begin{align*}
		\mathrm{x} _2(\Phi_x) \pi_\lambda (A-\lambda I)^{\ell}(L-A)
		&=\bigg(x\mathrm{w}_j\pi_\lambda+\mathrm{w}_j\sum_{j=0}^{d_\lambda-1}{-1\choose j}(\lambda-1)^{-1-j}N_\lambda^j\bigg) N_\lambda^{\ell}\pi_\lambda(L-A)\\
		&=\mathrm{w}_j\big(x+(\lambda-1)^{-1}\big) N_\lambda^{\ell}\pi_\lambda(L-A)\neq 0\\
	\end{align*}
	with positive probability, since $x+(\lambda-1)^{-1}\neq0$ because $\lambda\notin\R$. This implies that $\sigma_\ell^2(\Phi_x)>0$, and this proves \textit{i)}.
	
	\textit{ii)} Assume that $\mathrm{w}_jN^\ell_{\lambda}(L-A)=0$ almost surely for any $\lambda\in\sigma_A^{2}$ and any $\ell\ge 0$. 
Let $\mathrm{t}_j=\Phi_{0}(1)=a\mathrm{1}+b \mathrm{e}_j^\top L\in \R^{1\times J}$ be the random row vector whose expectation is $\E[\mathsf{t}_j]=\mathrm{w}_j\neq \mathrm{0}$ by assumption, and whose $i$-th entry is denoted by $\mathrm{t}_{ji}$.  Note that for $x\in(0,1)$, in view of \eqref{eq:definition of sigma}, we have
	\begin{align*}
		\sigma^2(\Phi_x)&=\sum_{k=0}^{\infty}\rho^{-k}\Var\left[\Phi_x(k)+\Psi^{\Phi_x}(k)\right]\mathrm{u}
		\ge\Var\left[\Phi_x(0)+\Psi^{\Phi_x}(0)\right]\mathrm{u}\\
		&\ge \E\big[\Var\big[\mathrm{t}_j\ind{U\le x}+\Psi^{\Phi_x}(0)\big|L\big]\mathrm{u}\big]+\Var\big[\E\big[\mathrm{t}_j\ind{U\le x}+\Psi^{\Phi_x}(0)\big|L\big]\mathrm{u}\big]\\
		&\ge \E\big[\Var\big[\mathrm{t}_j\ind{U\le x}\big|L\big]\mathrm{u}\big]=x(1-x)\E\Big[\sum_{i=1}^J \mathrm{t}^2_{ji} \mathrm u_i\Big]
		\ge x(1-x)\sum_{i=1}^J \mathrm{w}_{ji}^2 \mathrm u_i>0,
	\end{align*}
	since by assumption $\mathrm{w}_j\neq0$. %Above we have used the fact that $\E[\Phi^T_x(0)\mathrm{e}_i]=\E[\ind{U\leq x}]=x$ and $\Var[\Phi^T_x(0)\mathrm{e}_i]=x-x^2$. 
	Now we prove that $\sigma^2(\Phi_0)$ does not vanish.
	We have $\Phi_0(k)=\mathrm{w}_j\cdot  \ind{k\geq 1}$, so
	$\Phi_0:\N_0\mapsto \R^{1\times J}$ is completely deterministic. Assume that  $\sigma^2(\Phi_0)=0$.  Then for any $k\in \N_0$ it holds $\Var\left[\Psi^{\Phi_0}(k)\right]\mathrm{u}=0$ which in turn, as $\mathrm{u}_j>0$ for $j\in[J]$, implies that for any $k\in \N_0$ and $j\in[J]$ we have $\Var\left[\Psi^{\Phi_0}(k)\mathrm{e}_j\right]=0$. 
	The latter is equivalent  to 
	\begin{align}
		\label{eq:sigma is 0}
		\sum_{\ell\in \Z}\ind{k-\ell-1\ge 1}\mathrm{w}_jA^\ell\mathsf{P}(k,\ell)(L-A)\mathrm{e}_j=0\quad\text{almost surely}.
	\end{align}
	We set
	$A_\lambda = \pi_\lambda A+\lambda(I-\pi_\lambda)$, 
	and observe that for any  $n\in\Z$ and $m\in \N$ we have
	\begin{align*}
		A_{\lambda}^n\pi_{\lambda}=(\lambda I+N_\lambda)^n\pi_{\lambda}=\lambda^{n}\sum_{i=0}^{d_{\lambda}}\lambda^{-i}{n\choose i}N_{\lambda}^i\pi_{\lambda},\end{align*}
		and \begin{align*}
		\sum_{0\le \ell\le m}A_{1}^\ell\pi_{1}&=\sum_{0\le \ell\le m}(I+N_1)^ \ell\pi_{1}=\sum_{0\le \ell\le m}\sum_{i=0}^\ell{\ell\choose i}N_1^i\pi_{1}
		=\sum_{i=0}^{ d_1}N_1^i\sum_{l=i}^{m}{l\choose i}\pi_{1},
	\end{align*}
	where in the second equation we have used $\lambda=1$ (if $\lambda$ would be in the spectrum of $A$).
	Suppose now that for some vector $\mathrm{z}\in\R^J$  we have
	\begin{align*}
		\mathrm{w}_jN^i_\lambda\pi_\lambda \mathrm{z}=0\quad\text{for any  }\lambda\in\sigma^{2}_A \text{ and }\ i=0,\dots,d_\lambda-1,
	\end{align*}
	and for any $k\in\N_0$ it holds
	\begin{align}
		\label{eq:sigma is 0 with z}
		\sum_{\ell\in \Z}\ind{k-\ell-1\ge 1}\mathrm{w}_jA^\ell\mathsf{P}(k,\ell)\mathrm{z}=0.
	\end{align}
	Note that the left hand side of the previous equation can be rewritten as 
	\begin{align*}
		&\sum_{\ell\in \Z}\ind{k-\ell-1\ge 1}\mathrm{w}_jA^\ell\mathsf{P}(k,\ell)\mathrm{z}=
		\sum_{\lambda\in\sigma_A}\sum_{\ell\in \Z}\ind{k-\ell-1\ge 1}\mathrm{w}_jA_{\lambda}^\ell\pi_{\lambda}\mathsf{P}(k,\ell)\mathrm{z}\\
		&=\sum_{\lambda\in\sigma_A^{1}\cup\sigma_A^{3}}\sum_{\ell\in \Z}\ind{k-\ell-1\ge 1}\mathrm{w}_jA_{\lambda}^\ell\pi_{\lambda}\mathsf{P}(k,\ell)\mathrm{z}\\
		&=-\sum_{\lambda\in\sigma^{1}_A}\sum_{\ell\in \Z}\ind{k-\ell-1\ge 1}\mathrm{w}_jA_{\lambda}^\ell\pi_{\lambda}\ind{\ell<0}\mathrm{z}
		+\sum_{\lambda\in\sigma^{3}_A}\sum_{\ell\in \Z}\ind{k-\ell-1\ge 1}\mathrm{w}_jA_{\lambda}^\ell\pi_{\lambda}\ind{\ell\ge 0}\mathrm{z}.
	\end{align*}
	Setting  $-n=k-2\ge -2$, we get
\begin{align*}
		0&=\sum_{\lambda\in\sigma^{1}_A}\sum_{\ell\le -n}\mathrm{w}_jA_{\lambda}^\ell\ind{\ell<0}\pi_{\lambda}\mathrm{z}
		=\sum_{\lambda\in\sigma^{1}_A}\mathrm{w}_j(I-A_{\lambda}^{-1})^{-1}A_{\lambda}^{-n}\pi_{\lambda}\mathrm{z}\\
		&=\sum_{\lambda\in\sigma^{1}_A}\mathrm{w}_j(I-A_{\lambda}^{-1})^{-1}(\lambda\pi_{\lambda}+N_\lambda)^{-n}\pi_{\lambda}\mathrm{z}
		=\sum_{\lambda\in\sigma^{1}_A}\lambda^{-n}\sum_{i=0}^{d_{\lambda}}{-n\choose i}\lambda^{-i}\Big(\mathrm{w}_j(I-A_{\lambda}^{-1})^{-1}N_{\lambda}^i\pi_{\lambda}\mathrm{z}\Big),
\end{align*}
	which in view of Lemma~\ref{lem:linear independence}, implies that 
	\begin{align}
		\label{eq:sigma=0-1}
		\mathrm{w}_j(I-A_{\lambda}^{-1})^{-1}N^i_\lambda\pi_\lambda \mathrm{z}=0&\quad\text{for }\lambda\in\sigma_A^{1} \text{ and }\ i<d_\lambda.
	\end{align}
	Now we can use the decomposition
	$$(I-A_{\lambda}^{-1})^{-1}\pi_\lambda=\sum_{i=0}^{d_\lambda-1}c_iN^i_\lambda\pi_\lambda,$$
	for some $c_0,\dots,c_{d_\lambda-1}$ and since $(I-A_{\lambda}^{-1})^{-1}$ is invertible, we conclude that $(I-A_{\lambda}^{-1})^{-1}\pi_\lambda$ is not nilpotent and hence $c_0\neq0$. Taking now $i=d_{\lambda-1}$ in \eqref{eq:sigma=0-1}, we infer that $c_0\mathrm{w}_jN_\lambda^{d_{\lambda}-1}\mathrm{z}=0$, that is, we have $\mathrm{w}_jN_{\lambda}^{d_{\lambda}-1}\mathrm{z}=0$. Recursively, we see that  for $i=d_{\lambda}-1,d_{\lambda}-2,\dots,0$ equation \eqref{eq:sigma=0-1} implies
$\mathrm{w}_jN_\lambda^i\mathrm{z}=0$, for any $\lambda\in\sigma_A^{1}$. 
	Taking this into account and setting $n=k-2>0 $ condition \eqref{eq:sigma is 0 with z} gives
	\begin{align*}
		0&=\sum_{\lambda\in\sigma^3_A}\sum_{0\le \ell\le n} \mathrm{w}_jA_{\lambda}^\ell\pi_{\lambda}\mathrm{z}
		=\sum_{\lambda\in\sigma^3_A\setminus\{1\}}\sum_{0\le \ell\le n} \mathrm{w}_jA_{\lambda}^\ell\pi_{\lambda}\mathrm{z}
		+\ind{1\in\sigma^{3}_A}\sum_{0\le \ell\le n} \mathrm{w}_jA_{1}^\ell\pi_{1}\mathrm{z}\\
		&=\sum_{\lambda\in\sigma^3_A\setminus\{1\}} \mathrm{w}_j(A_{\lambda}-I)^{-1}(A_{\lambda}^{n+1}-I)\pi_{\lambda}\mathrm{z}
		+\ind{1\in\sigma^{3}_A}\sum_{0\le \ell\le n} \mathrm{w}_j(I+N_1)^\ell\pi_{1}\mathrm{z}\\
		&=\sum_{\lambda\in\sigma^3_A\setminus\{1\}}\sum_{i=0}^{d_\lambda-1}\mathrm{w}_j(A_{\lambda}-I)^{-1}N^i_\lambda\pi_\lambda \mathrm{z}\lambda^np_{\lambda,i}(n)+\ind{1\in\sigma^{3}_A} \sum_{i=0}^{d_1-1}\mathrm{w}_jN^i_1\pi_1 \mathrm{z} p_{1,i+1}(n)+c,
	\end{align*}
where $p_{\gamma,i}$ is some polynomial of degree $i$ and $c$ does not depend on $n$.
Lemma~\ref{lem:linear independence} implies that
	\begin{align*}
		\mathrm{w}_j(A_{\lambda}-I)^{-1}N^i_\lambda\pi_\lambda \mathrm{z}=0\quad\text{for }\lambda\in\sigma_A^{3}\setminus\{1\} \text{ and }i<d_\lambda,
		\end{align*}
and also 
\begin{align*}
		\mathrm{w}_jN^i_1\pi_1 \mathrm{z}=0\quad\text{if  }1\in\sigma_A^{3} \text{ and } i<d_1.
	\end{align*}
	The same argument as before gives that $\mathrm{w}_jN_\lambda^i\mathrm{z}=0,$ for any $\lambda\in\sigma_A^{3}$ and $i<d_\lambda$. 
	Suppose now that $\sigma^2(\Phi_0)=0$ and also $\sigma_\ell^2(\Phi_0)=0$ for all $0\leq \ell\leq J$. Then by setting $\mathrm{z}=(L-A)\mathrm e_j$ with $j\in [J]$ we conclude that
	for any $\lambda\in \sigma_A$ and $\ell\ge 0$ it holds  
	\begin{align*}
		\mathrm{w}_jN^\ell_\lambda\pi_\lambda L=\mathrm{w}_jN^\ell_\lambda\pi_\lambda A\quad \text{almost surely},
	\end{align*}
thus contradicting the assumption.
\end{proof}

\subsubsection*{Continuity of the limit processes $\Gt$ and $\Gj$}

We recall once again the notation $\Gphi(x)=\rho^{\lfloor x\rfloor/2}G^{\Phi_{\{x\}}}$, where  
$G^{\Phi_{\{x\}}}\eqdist\sigma_\ell(\Phi_{\{x\}})\mathcal{N}$ or $G^{\Phi_{\{x\}}}\eqdist\sigma(\Phi_{\{x\}})\mathcal{N}$
and $\mathcal{N}:=\mathcal{N}(0,1)$ is a standard normal variable independent of $W$, and $\Gt$ (respectively  $\Gphi$) stands for $\Gphi$ if $\Phi=\Phi^{\mathsf{t}}$ (respectively $\Phi=\Phi^i$). 
\begin{lemma}\label{lem:properties of the gaussian process}
	For  $j\in[J]$, let $\mathcal{H}(x)=\rho \mathrm{u}_j \Gt(x)-\Gj(x)$. Under the assumptions of Theorem \ref{thm:clt-general}, $\mathcal{H}(x)$
 is continuous for any $x\in [0,\infty)$.
\end{lemma}
\begin{proof}
%It is enough to prove that $\mathcal{H}(x)$ is continuous on $[0,1)$ and then the claim follows from Kolmogorov continuity theorem once we show that
%$$\E[|\mathcal{H}(x)-\mathcal{H}(y)|]\le C|x-y|,$$
%for some constant $C$ and any $x,y\in[0,1)$.
Since $\mathcal{H}$ is a linear combination of the Gaussian processes $\Gt$ and $\Gj$, it is enough to show continuity for both terms separately. We start with the continuity of $\Gj$. For any  $0\le x\le y\le 1$, since either 
$$\Gj(y)-\Gj(x)\eqdist(\sigma(\Phi^j_{y})-\sigma(\Phi^j_x))\mathcal{N}=\sigma(\Phi^j_{y}-\Phi^j_{x})\mathcal{N}$$
 or 
 $$\Gj(y)-\Gj(x)\eqdist(\sigma_\ell(\Phi^j_{y})-\sigma_\ell(\Phi^j_x))\mathcal{N}=\sigma_\ell(\Phi^j_{y}-\Phi^j_{x})\mathcal{N},$$ 
 depending on whether we are in case {\it i)} or {\it ii)} of Theorem 3.5 of  \cite{kolesko-sava-huss-char} respectively, we have to upper bound $\sigma^2(\Phi^j_{y}-\Phi^j_{x})$ and $\sigma_\ell^2(\Phi^j_{y}-\Phi^j_{x})$ by some power of $|y-x|$. The definition of $\sigma^2(\Phi)$ applied  to the characteristic $\Phi^j_{y}-\Phi^j_{x}$ yields:
\begin{align*}
	&\sigma^2(\Phi^j_{y}-\Phi^j_{x})=\sum_{k<0}\rho^{-k}\Var[\mathrm{e}_j\pi^{(1)}(y-x)A^{k-1})(L-A)]\mathrm{u}	\\+ &\Var[\mathrm{e}_j\ind{x< U\le y}-\mathrm{e}_j\pi^{(1)}(y-x)A^{-1}(L-A)]\mathrm{u}	
	+\rho\Var[(y-x)e_j\pi^{(3)}(L-A)]\mathrm{u}	\le C(y-x).
	\end{align*}
On the other hand, as for $\Phi^j_{y}-\Phi^j_{x}$ it holds $\mathrm{x}_2=\mathrm{x}_2(\Phi^j_{y}-\Phi^j_{x})=(y-x)\mathrm{e}_j\pi^{(2)}$,
we conclude that
	\begin{align*}
	\sigma_\ell^2(\Phi^j_{y}-\Phi^j_{x})= \frac{\rho^{-\ell}}{(2\ell+1)(\ell!)^2}\sum_{\lambda\in\sigma^2_A}\Var \left[\mathrm{x} _2 \pi_\lambda (A-\lambda I)^{\ell}L\right]\mathrm{u}\le C(y-x)^2.
	\end{align*}
In particular, in both of the cases {\it i)} and {\it ii)} of \cite[Theorem 3.5]{kolesko-sava-huss-char} we  have 
	$$\E[|\Gj(y)-\Gj(x)|^{2}]\le C(y-x),$$
	and therefore, since $\Gj$ is Gaussian, we obtain 
	$$\E[|\Gj(y)-\Gj(x)|^{4}]=6\E[|\Gj(y)-\Gj(x)|^{2}]^2\le C|y-x|^2,$$ 
	which by Kolmogorov continuity theorem implies that $\Gj$ is continuous.
	%\kk{In fact, we can show that $G^j$ is a gaussian bridge between $G(0)$ and $G(1)$ in the case i) and the linear function in case ii)} 
The same calculations as for $\Gj$ can be done in order to prove that $\Gt$ is continuous, so also $\mathcal{H}(x)=\rho\mathrm{u}_j\Gt(x)-\Gj(x)$ is continuous.
\end{proof}

\subsubsection*{Localization of the stopping times}	
	
This section is dedicated to the localization of the stopping times $(\tau_k)_{k\in\N}$. On the non-extinction event $\S$, for any $n\in\N$, we define the random variable 
 $$T_n=\log_\rho\frac{n(\rho-1)}{W},$$
and the function $h:\R\to\R$ by
 \begin{equation}\label{eq:h-function}
 h(x)= \floor x+\frac{\rho^{\{x\}}-1}{\rho-1}=x+\frac{\rho^{\{x\}}-1}{\rho-1}-\{x\}.
 \end{equation}
 Note that $h^{-1}$ is uniformly continuous and given by $h^{-1}(x)=\floor x+\log_\rho\big(1+(\rho-1)\{x\}\big)$.
 \begin{proposition}\label{prop:loc-stop-times}
 Under the assumptions (GW1)-(GW3), if for $k\in\N$ we write $t_k=h(T_k)$, then 
%For $k\in\N$ we define 
%\begin{align}\label{eq:definition of t_k}
%t_k\defeq \lfloor W_k \rfloor+\dfrac{\rho^{\{W_k\}}-1}{\rho-1}=W_k+\dfrac{\rho^{\{W_k\}}-1}{\rho-1}-\{W_k\}.
%\end{align}	
for $(\tau_k)$ defined as in \eqref{eq:tau-urn}, we have
$$\lim_{k\to\infty}(t_k-\tau_k)=\lim_{k\to\infty}(h(T_k)-\tau_k)=0,\quad\ProbS\textit{-almost surely}.$$
\end{proposition}
\begin{proof}
By Proposition \ref{lln}, the following $\ProbS$-almost sure convergence holds:
$$\lim_{x\to\infty}\frac{\Zt(x)}{\rho^{\lfloor x\rfloor}(1+(\rho-1)\{x\})}=\lim_{x\to\infty}\frac{\Zt(x)}{\rho^{x}l_\rho(x)}=\frac{1}{\rho-1}W.$$
We recall that $l_{\rho}:[0,\infty)\to\R$ is defined as $l_{\rho}(x)=(1+(\rho-1)\{x\})\rho^{-\{x\}}$.
Since $\Zt(\tau_k)=k$, we infer that for any $\delta>0$ and large enough $k$ we have
\begin{align*}
\frac{k}{\rho^{\tau_k}l_\rho(\tau_k)}\le \frac{1}{\rho-1}We^{\delta}
\quad \text{and}\quad 
 \frac{k}{\rho^{\tau_k-\delta}l_\rho(\tau_k-\delta)}\ge \frac{1}{\rho-1} We^{-\delta}.
 \end{align*}
 This can be rewritten as 
 $$ \tau_k +\log_{\rho}l_{\rho}(\tau_k-\delta)-\log_{\rho}e^{-\delta}\leq T_k \leq  \tau_k +\log_{\rho} l_{\rho}(\tau_k)+\log_{\rho}e^{\delta}. $$
 Remark that we have
 $$ \tau_k+\log_{\rho} l_{\rho}(\tau_k)= \lfloor \tau_k\rfloor + \log_{\rho}(1+(\rho-1)\{\tau_k\})=h^{-1}(\tau_k),$$  
and the inverse of the increasing function $h^{-1}(x)$ is given by $ \lfloor x\rfloor +\frac{\rho^{\{x\}}-1}{\rho-1}=h(x)$
which then yields the following inequalities
\begin{align*}
\lfloor T_k -\log_{\rho}e^{\delta}\rfloor+\dfrac{\rho^{\{T_k-\log_{\rho}e^{\delta}\}}-1}{\rho-1}\le\tau_k
 &\le 
 \lfloor T_k -\log_{\rho}e^{-\delta}\rfloor+\dfrac{\rho^{\{T_k-\log_{\rho}e^{-\delta}\}}-1}{\rho-1}+\delta.
 \end{align*}
 	From the uniform continuity of $h(x)$, by letting $\delta\to 0$, we obtain the claim.
 \end{proof}
The above proposition implies that 
$$\tau_n=\log_{\rho}n+O(1),\quad\ProbS\text{-almost surely.}$$
 
\subsection{Limit theorems for $B_j$}

%\textcolor{blue}{The proof below should be rewritten!}
\begin{proposition}
	\label{lem:clt for fluctuations}
Under the assumptions of Theorem \ref{thm:clt-general}, let $\Phi:\Z\to\R^{1\times J}$ be any characteristic such that the following stable convergence holds 
\begin{align}
		\label{eq:assumption FLT}
	\frac{\Xphi(n+x)}{n^{\ell+\frac12}\rho^{n/2}\sqrt W}\stabSto \Gphi(x)\quad\text{in }\mathcal D(\R),
\end{align}
 for some continuous Gaussian process $\Gphi$ with $\Var \left[\Gphi(x)\right]>0$, for any $x\in\R$.  Then there exists a continuous, positive, 1-periodic function $\Uppsi^{\Phi}$ such that for $\tau_n$ as in \eqref{eq:tau-urn} and $T_n=\log_\rho\frac{n(\rho-1)}{W}$ it holds
	\begin{align}
		\label{eq:fluctuations of X(tau n)}
		\frac{\Xphi(\tau_n)}{ \sqrt n (\log_\rho n)^{l+\frac12}\Uppsi^{\Phi}(T_n) }\distSto \cN(0,1),\quad \text{ as } n\to\infty.
	\end{align}
\end{proposition}
\begin{proof}  
The key idea in the proof is to use the functional limit theorem for $\Xphi$ in order to replace $\Xphi(\tau_n)$ by $\Xphi(t_n)$. Recall that for $h$ defined in \eqref{eq:h-function}, which is continuous and strictly increasing we have used the notation $t_n=h(\log_\rho n-\log_\rho\frac{W}{\rho-1})=h(T_n)$. Furthermore, for any $x\in\R$ and $n\in\N$ it holds $h(n+x)=n+h(x)$. Consequently, by \eqref{eq:assumption FLT}, we have
	\begin{align*}
		\frac{\Xphi \big(h(n+x)\big)}{n^{\ell+\frac12}\rho^{n/2}\sqrt W}=\frac{\Xphi \big(n+h(x)\big)}{n^{\ell+\frac12}\rho^{n/2}\sqrt W}\stabSto \Gphi \big(h(x)\big)\quad\text{in }\mathcal D(\R),
	\end{align*}
and the latter convergence can be rewritten as
	\begin{align*}
			\frac{\Xphi \big(h(n+x)\big)}{\Big(n^{2\ell+1}\rho^{n} W\Var \left[\Gphi\big(h(x)\big)\right]\Big)^{1/2}}\stabSto G(x)\quad\text{in }\mathcal D(\R),
	\end{align*}
for some stationary and continuous process $G$ with $G(0)\eqdist\mathcal N(0,1)$.
Note that one consequence of the convergence in \eqref{eq:assumption FLT} is the following property of the limiting process: $\Gphi(x+1)\eqdist\sqrt\rho \Gphi(x)$. Hence, the previous convergence is equivalent to 
\begin{align*}
	\frac{\Xphi \big(h(n+x)\big)}{\Big(\big((n+x)\vee 1\big)^{2\ell+1}\rho^{n+x}\rho^{-\{x\}}W\Var \left[\Gphi\big(h(\{x\})\big)\right]\Big)^{1/2}}\stabSto G(x)\quad\text{in }\mathcal D(\R).
\end{align*}
In other words, for the function $\Uppsi^\Phi$ defined by  $\Uppsi^\Phi(x)=\Big((\rho-1)\rho^{-\{x\}} \Var \left[\Gphi\big(h(\{x\})\big)\right]\Big)^{1/2}$ which is continuous and positive, and for
	\begin{align*}
		X(x)=\frac{\Xphi \big(h(x)\big)}{\Big(\big(x\vee 1\big)^{2\ell+1}\rho^{x}\frac{W}{\rho-1}\Big)^{1/2} \Uppsi^\Phi(x)}
	\end{align*}
it holds $X(n+x)\stabSto G(x)$, and therefore an application of Lemma \ref{lem:auxiliary lemma for cadlag processes}\textit{ ii)} with $a_n= \log_\rho n$ yields 
\begin{align}
	\label{eq:convergence of Xtn}
	X\big(h^{-1}(t_n)\big)=X\big(\log_\rho n-\log_\rho\tfrac{W}{\rho-1}\big)\distSto G(0).
\end{align}
Since $h^{-1}(x)$ is uniformly continuous, by Proposition~\ref{prop:loc-stop-times} we get
\begin{align}
	\label{eq:comparing tn and taun}
	h^{-1}(\tau_n)-h^{-1}(t_n)\to 0,\quad\ProbS\text{-almost surely}.
\end{align}

We claim that from Lemma \ref{lem:auxiliary lemma for cadlag processes}\textit{ i)} with $N_n= \floor{\log_{\rho} n}$ and $\delta_n=2^{-n}$ it follows that 
\begin{align}
	\label{eq:comparision of X(tn) and X(taun)}
	X\big(h^{-1}(\tau_n)\big)-X\big(h^{-1}(t_n)\big) \ProbSto0, \quad \text{as } n\to\infty.
\end{align}
Indeed, for fixed $m\in\N$, on the event $|\log_{\rho} W|\le k_m-2\delta_m-1-\log_\rho(\rho-1)$ and  $|h^{-1}(\tau_n)-h^{-1}(t_n)|\le \delta_m$ we have 
$$\big|X\big(h^{-1}(\tau_n)\big)-X\big(h^{-1}(t_n)\big)\big|\le 	\omega(X_{N_{m}},k_m,\delta_{m}).$$
In turn, for any $\varepsilon>0$ we get
\begin{align*}
	&\ProbS\Big(\big|X\big(h^{-1}(\tau_n)\big)-X\big(h^{-1}(t_n)\big)\big|>\varepsilon\Big)\le \ProbS\big(	\omega(X_{N_{m}},k_m,\delta_{m})>\varepsilon\big)\\
	&+\ProbS\big(|\log_\rho W|> k_m-2\delta_m-1-\log_\rho(\rho-1)\big)
	+\ProbS\big(|h^{-1}(\tau_n)-h^{-1}(t_n)|> \delta_m\big).
\end{align*}
Taking first the limit as $n\to\infty$ and then as $m\to\infty$ and using \eqref{eq:comparing tn and taun}, we get \eqref{eq:comparision of X(tn) and X(taun)}.
Finally, \eqref{eq:comparision of X(tn) and X(taun)} and \eqref{eq:convergence of Xtn} imply that $$\frac{\Xphi(\tau_n)}{\big(\kk{h^{-1}(\tau_n)}\vee 1\big)^{\ell+\frac12}\sqrt n\rho^{(h^{-1}(\tau_n)-h^{-1}(t_n))/2} \Uppsi^\Phi(h^{-1}(\tau_n))}=X\big(h^{-1}(\tau_n)\big)\distSto G(0),$$
which together with	\eqref{eq:comparing tn and taun} and the $\ProbS$-almost sure convergence   of 
$\frac{\Uppsi^\Phi(h^{-1}(\tau_n))}{\Uppsi^\Phi(h^{-1}(t_n))}$ and $\frac{h^{-1}(\tau_n)}{\log_\rho n}$ to 1 yield that
\begin{align*}
	\frac{\Xphi(\tau_n)}{\sqrt n(\log_\rho n)^{\ell+\frac12} \Uppsi^\Phi(T_n)}\distSto \mathcal N(0,1),
\end{align*}
that is \eqref{eq:fluctuations of X(tau n)} holds and this finalizes the proof.
\end{proof}
Using the auxiliary result that we just proven, we can now state and prove our main result.
\begin{theorem}
	\label{thm:general_limit_theorem}
Suppose that \eqref{eq:condition for nondegenerate sigma} holds and all assumptions from Theorem \ref{thm:clt-general} are satisfied. Then there exists a continuous, positive and $1$-periodic function $\Uppsi$ such that, for $T_n=\log_{\rho}\frac{n(\rho-1)}{W}$  and any $j\in[J]$ we have
	\begin{align*}
		\frac{B_j(n)-\Fj\big(\Fi(n)\big)}{\sqrt n(\log_{\rho} n)^{\ell+\frac12}\Uppsi(T_n)}\distSto \mathcal{N}(0,1), \quad \text{as }n\to\infty.
	\end{align*}
\end{theorem}
\begin{proof}
Since for any $j\in[J]$,	we have $\Zj(\tau_n)=B_j(n)$ and $\Xj(n)=\Zj(n)-\Fj(n)$ , we can write
\begin{align*}
B_j(n)&=\Fj(\tau_n)+\Xj(\tau_n)\\
&=\Fj\big(\Fi(n)\big)+\Xj(\tau_n)+(\Fj\circ \Fi\circ \Ft(\tau_n)-\Fj(\Fi(n))).
\end{align*}
Due to the fact that $\Ft(\tau_n)=n-\Xt(\tau_n)$ together with Lemma \ref{lem:difference of composition of F}, we obtain
\begin{align*}
\Fj\circ \Fi\circ \Ft(\tau_n)-\Fj\circ \Fi(n)+\rho \mathrm{u}_j  \Xt(\tau_n)=\Xt(\tau_n)o(1),\quad \ProbS\text{-almost surely},
\end{align*}
and therefore it holds 
\begin{align*}
B_j(n)-\Fj(\Fi(n))=\Xj(\tau_n)-\rho  \mathrm{u}_j\Xt(\tau_n)+o(1)\Xt(\tau_n)\quad \ProbS\text{-almost surely}.
\end{align*}
By Proposition \ref{prop:pos-variances} (positivity of the variances of the limiting process) and Lemma \ref{lem:properties of the gaussian process} (continuity of the limit processes)  the characteristics $\Phi=\Phi^j-\rho\mathrm{u}_j\Phi^{\mathsf{t}}$ and $\Phi=\Phi^{\mathsf{t}}$ and the corresponding processes $\Xphi$ with this characteristic satisfy the assumptions of Proposition \ref{lem:clt for fluctuations},  so we can apply Proposition \ref{lem:clt for fluctuations} to the processes $\Xphi=\Xj-\rho\mathrm{u}_j\Xt$ and to $\Xt$ in order to obtain
\begin{align*}
\frac{\Xj(\tau_n)-{\rho}\mathrm{u}_j\Xt(\tau_n)}{ \sqrt n (\log_\rho n)^{\ell+\frac12}\Uppsi(T_n)} \distSto\cN(0,1),
\end{align*} for some function $\Uppsi$ which is continuous, positive, and $1$-periodic, and
$\frac{\Xt(\tau_n)}{ \sqrt n (\log_\rho n)^{\ell+\frac12} }\cdot o(1) \Probto 0$ as $n\to\infty$, and this completes the proof.
\end{proof}

\vspace{0.2cm}
Finally, in view of Theorem \ref{thm:general_limit_theorem} it suffices to find an expansion of $\Fj(\Fi(n))$ up to an error of order $o(n^{\log_\rho\gamma})$
in order to prove Theorem \ref{them:main_CLT}, which will then follow immediately from the next corollary.

\begin{corollary}
Under the assumptions of Theorem \ref{thm:general_limit_theorem} suppose that all eigenvalues in $\Gamma$ are simple. Then the following holds:
	\vspace{-0.3cm}
	\begin{enumerate}[i)]\setlength\itemsep{-0.2em}
		\item If $\gamma>\sqrt \rho$, then for any $\lambda\in\Gamma$ there is a $1$-periodic, continuous function $f_\lambda:\R\to\C$ and a random variable $X_\lambda$ such that
		$$B_j(n)=\rho\mathrm{u}_j\cdot n+\sum_{\lambda\in \Gamma}n^{\log_\rho\lambda}f_\lambda(T_n)X_\lambda+o_{\Prob}\Big(n^{\log_\rho\gamma}\Big)$$
		\vspace{-0.5cm}
		\item If $\gamma=\sqrt \rho$, then there is a $1$-periodic, continuous function $\Uppsi:\R\to(0,\infty)$  such that the following convergence holds:
		$$\frac{B_j(n)-\rho\mathrm{u}_j\cdot n}{\sqrt n(\log_{\rho} n)^{\frac12}\Uppsi(T_n)}\distSto \mathcal{N}(0,1),\quad \text{as }n\to\infty.$$
		\vspace{-0.5cm}
		\item If $\gamma<\sqrt \rho$, then there is a $1$-periodic, continuous function $\Uppsi:\R\to(0,\infty)$  such that the following convergence holds:
		$$\frac{B_j(n)-\rho\mathrm{u}_j\cdot n}{\sqrt n\Uppsi(T_n)}\distSto \mathcal{N}(0,1),\quad \text{as }n\to\infty.$$
	\end{enumerate} 
\end{corollary}
\begin{proof} We handle case \textit{i)} in detail, and the other two cases are identical so we leave the details to the interested reader.
In case $\gamma>\sqrt \rho$ by \eqref{eq:fluctuations-phi} we have
\begin{align*}
	\Fphi(n)&=	\mathrm{x}_1(\Phi)A^n_1 W^{(1)}+\mathrm{x}_2(\Phi)A^n_2Z_0
	=	\sum_{\lambda\in\Gamma\cup\{\rho\}}\lambda^n\mathrm{x}_1(\Phi)\pi_{\lambda}W^{(1)}+o(\gamma^n)\\
	&=\rho^n\mathrm{x}_1(\Phi) W\mathrm u+\sum_{\lambda\in\Gamma}\lambda^n\mathrm{x}_1(\Phi) W_\lambda \mathrm u^{\lambda}+o(\gamma^n),
\end{align*}
where we define $W_\lambda \mathrm u^\lambda=\pi_\lambda W^{(\lambda)}$ for some scalar random variable $W_\lambda$ (this can be done as $\pi_\lambda$ is a projection on the space spanned by the eigenvector $\mathrm u^\lambda$). In particular, as $\Fj$ and $\Ft$ are piecewise linear between consecutive integer arguments we conclude that
\begin{align*}
	\Fj(x)&=\rho^xl_\rho(x)\mathrm{x}_1(\Phi_{0}^j) W\mathrm u+\sum_{\lambda\in\Gamma}\lambda^xl_\lambda(x)\mathrm{x}_1(\Phi_0^j) W_\lambda \mathrm u^{\lambda}+o(\gamma^x),
\end{align*}
and 
taking into account that $\mathrm{x}_1(\Phi_{0}^j)= \frac{\lambda}{\lambda-1}\mathrm{e}_j^\top\pi_\lambda$  we finally get
\begin{align*}
	\Fj(x)&=\frac{\rho^{x+1}}{\rho-1}l_\rho(x)  W\mathrm{u}_j+\sum_{\lambda\in\Gamma}\frac{\lambda^{x+1}}{\lambda-1}l_\lambda(x) W_\lambda \mathrm{u}_j^\lambda+o(\gamma^x).
\end{align*}
By the same argument, we also obtain
\begin{align*}
	\Ft(x)=\sum_{j=1}^{J} \Fj(x-1)=\frac{\rho^x}{\rho-1}l_\rho(x)  W+\sum_{\lambda\in\Gamma}\frac{\lambda^x}{\lambda-1}l_\lambda(x) W_\lambda \Big(\sum_{i=1}^{J}\mathrm{u}_{i}^\lambda\Big) +o(\gamma^x).
\end{align*}
In particular, it holds
\begin{align*}
	\F^j(x)=\rho \mathrm{u}_j\cdot \Ft(x)+\sum_{\lambda\in\Gamma}\frac{\lambda^{x}}{\lambda-1}l_\lambda(x)\Big(\lambda\mathrm{u}_j^\lambda-\rho \mathrm u_j\Big(\sum_{i=1}^J\mathrm{u}_i^\lambda\Big)\Big) W_\lambda+o(\gamma^x).
\end{align*}
Since for $\lambda\in \Gamma$ we have
 $$\lambda^{x}=\bigg(\frac{\rho-1}{l_\rho(x) W}\Ft(x)\bigg)^{\log_\rho \lambda}(1+o(1))=\bigg(\frac{\rho-1}{l_\rho(x) W}\Ft(x)\bigg)^{\log_\rho \lambda}+o(\gamma^x) \text{ on } \Surv,$$ 
 and  
 $\Fi(n)=t_n+o(1)=h(T_n)+o(1)$, we deduce that
\begin{align*}
	\Fj(\Fi(n))&
%	=\rho \mathrm{u}_jn+\sum_{\lambda\in\Gamma}n^{\log_\rho \lambda} \frac{l_\lambda(t_n)}{l_\rho(t_n)^{\log_\rho \lambda}}\\
%	&\qquad\times\Big(\lambda\mathrm{u}^j_\lambda-\rho \mathrm u^j\Big(\sum_i\mathrm{u}_i^\lambda\Big)\Big)\Big(\frac{\rho-1}{ W}\Big)^{\log_\rho \lambda} \frac{W_\lambda}{\lambda-1}+o(\gamma^{t_n})\\
%	&
	=\rho \mathrm{u}_j\cdot n+\sum_{\lambda\in\Gamma}n^{\log_\rho \lambda} \frac{l_\lambda(h(T_n))}{l_\rho(h(T_n))^{\log_\rho \lambda}}
	\cdot\Big(\lambda\mathrm{u}_j^\lambda-\rho \mathrm u_j\Big(\sum_{i=1}^J \mathrm{u}_i^\lambda\Big)\Big)\Big(\frac{\rho-1}{ W}\Big)^{\log_\rho \lambda} \frac{W_\lambda}{\lambda-1}\\
	&+o_{\mathbb{P}}(n^{\log_\rho\gamma}),
\end{align*}
and thus \textit{i)} holds with 
$$f_\lambda(x)=\frac{l_\lambda(h(x))}{l_\rho(h(x))^{\log_\rho \lambda}}\quad \text{ and} \quad X_\lambda= \big(\lambda\mathrm{u}_j^\lambda-\rho \mathrm u_j\big(\sum_{i=1}^J\mathrm{u}_i^\lambda\big)\big)\Big(\frac{\rho-1}{ W}\Big)^{\log_\rho \lambda} \frac{W_\lambda}{\lambda-1},$$
 for every $\lambda\in \Gamma$.

In the case $\gamma=\sqrt \rho$ we have
\begin{align*}
	\F^{\Phi}(n)=\rho^n\mathrm{x}_1(\Phi) W \mathrm u +\sum_{\lambda\in\Gamma}\lambda^n\mathrm{x}_1(\Phi) \pi_\lambda Z_0+o(\gamma^n)=\rho^n\mathrm{x}_2(\Phi)W \mathrm u +o(\rho^{n/2}),
\end{align*}
and for $\gamma<\sqrt \rho$
\begin{align*}
	\F^{\Phi}(n)=\rho^n\mathrm{x}_1(\Phi) W \mathrm u+o(\rho^{n/2}).
\end{align*}
In both cases, the same approach as in the proof of part \textit{i)}, after an application of Theorem \ref{thm:general_limit_theorem}, proves finally \textit{ii)} and \textit{iii)}.
\end{proof}

\appendix
\section{Appendix A}\label{sec:apA}

\begin{figure}
\centering
\begin{tikzpicture}[scale=0.7]
\node[shape=circle,draw=black,fill=black] at (0,0){};
\draw (10,0) node {$\mathsf{U}_1$};

\node[shape=circle,draw=black,fill=black] at (-2,-1.5){};
\node[shape=circle,draw=green,fill=green] at (4,-1.5){};
\draw (10,-1.5) node {$\mathsf{U}_2$};

\node[shape=circle,draw=black,fill=black] at (-3,-3){};
\node[shape=circle,draw=green,fill=green] at (-1,-3){};
\node[shape=circle,draw=black,fill=black] at (3,-3){};
\node[shape=circle,draw=black,fill=black] at (1,-3){};
\node[shape=circle,draw=red,fill=red] at (5,-3){};
\node[shape=circle,draw=green,fill=green] at (7,-3){};
\draw (10,-3) node {$\mathsf{U}_1$};

\node[shape=circle,draw=red,fill=red] at (5,-4.5){};
\node[shape=circle,draw=red,fill=red] at (3,-4.5){};
\node[shape=circle,draw=green,fill=green] at (7,-4.5){};
\draw (10,-4.5) node {$\mathsf{U}_2$};

\begin{pgfonlayer}{bg}
\draw[color=black] (0,0) -- (-2,-1.5);
\draw[color=black] (0,0) -- (4,-1.5);
\draw[color=black] (4,-1.5) -- (3,-3);
\draw[color=black] (4,-1.5) -- (1,-3);
\draw[color=black] (4,-1.5) -- (5,-3);
\draw[color=black] (4,-1.5) -- (7,-3);
\draw[color=black] (-2,-1.5) -- (-3,-3);
\draw[color=black] (-2,-1.5) -- (-1,-3);
\draw[color=black] (5,-3) -- (5,-4.5);
\draw[color=black] (5,-3) -- (3,-4.5);
\draw[color=black] (5,-3) -- (7,-4.5);
\end{pgfonlayer}
\end{tikzpicture}
\caption{The model with two alternating urns and deterministic replacement matrix after $n=4$ draws.}
\label{fig:urn}
\end{figure}
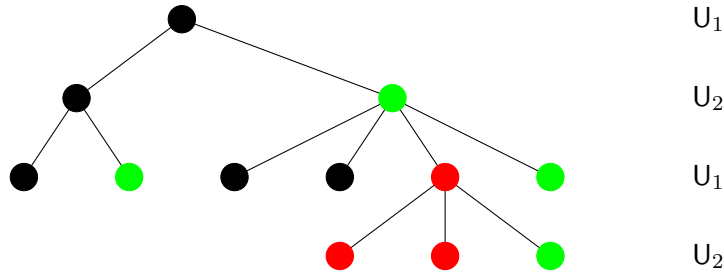

We illustrate here the model with two alternating urns on an example with $J=3$ colours (1=black, 2=red, and 3=green in the tree from Figure \ref{fig:urn}) and deterministic replacement matrix $L$ given by
\begin{equation*}
L=
\begin{pmatrix}
1 & 0 & 2\\
0 & 2 & 1\\
1 & 1 & 1
\end{pmatrix}
\end{equation*}
and $j_0=1$, that is we start with one black ball in $\mathsf{U}_1$ at time $0$, so $B(0)=(1,0,0)$. The first column $L^{(1)}$ of $L$ tells us that when picking a black ball from some urn, one black and one green ball is added to the other urn, so $\mathsf{U}_2$ is given by the nodes at level one of the tree and $B(1)=(2,0,1)$. Since after one step the content of $\mathsf{U}_1$ has been emptied, we proceed by picking up balls step-by-step from $\mathsf{U_2}$ (from the first level of the tree). After picking a green ball from $\mathsf{U}_2$, since the third column of the matrix $L$ gives the number and the colours of balls added to the other urn, with probability $1/2$ the number of added balls after $2$ steps is $B(2)=(4,1,2)$ and with probability $1/2$ is $B(2)=(3,0,2)$. Then $B(3)=(5,1,3)$ and $\mathsf{U}_2$ has been emptied, so we proceed again to $\mathsf{U}_1$ which contains now $6$ balls, and with probability $1/6$ we have $B(4)=(5,3,4)$; now we have started to build the third level of the random tree and to refill again $\mathsf{U}_2$.

\section{Appendix B}\label{sec:apB}

\paragraph{Higher moments estimates for $Z_n^{\Phi}$.} We provide here, for a general random characteristic $\Phi$, higher moments estimates of the random variable ${Z_n^{\Phi}-\mathrm{x}_1A^n_1 W^{(1)}-\mathrm{x}_2A^n_2Z_0}$, needed in the proof of Theorem \ref{thm:clt-general} for the characteristics $\Phi^{\mathsf{t}}$ and $\Phi^j$ which fulfill the assumptions of the next result.
Recall that $\Phi$ is called \textit{centered} if $\E[\Phi(k)]=\mathrm{0}\in \R^{1\times J}$ for any $k\in\Z$.  
\begin{theorem}
	\label{lem:moments growth}
	Let $p\in[1,2]$ and $\Phi:\Z\to \C^{1\times J}$ be a random characteristic. Moreover, assume that (GW1)-(GW3) hold and the second moment of $L$ is finite. 
	\vspace{-0.3cm}
	\begin{enumerate}[i)]
		\item If $\sum_{k\in \Z}\big(\E\big[ \|\Phi(k)\|^p\big]\big)^{1/p} \rho^{-k}  <\infty$, then $\E\big[|Z^\Phi_n|^{p}\big]=O(\rho ^{pn})$. 
		\item If $\sum_{k\leq n}\big(\E\big[ \|\Phi(k)\|^p\big]\big)^{1/p} \rho^{-k}  =O(n^r) $ for some $r\ge0$, then $\E\big[|Z^\Phi_n|^p\big]=O(\rho ^{pn}n^{pr})$.
	\end{enumerate}
If additionally $\Phi$ is centered  then the following holds.
	\begin{enumerate}[i)]
		\setcounter{enumi}{2}
		\item If $\sum_{k\in \Z}\big(\E\big[ \|\Phi(k)\|^{2p}\big]\big)^{1/p}\rho^{-k}<\infty$, then $\E\left[|Z^\Phi_n|^{2p}\right]=O(\rho^{np})$.
		\item If  $\sum_{k\le n}\big(\E\big[ \|\Phi(k)\|^{2p}\big]\big)^{1/p}\rho^{-k}=O(\rho^nn^r)$, then $\E\left[|Z^\Phi_n|^{2p}\right]=O(\rho^{np}n^{pr})$.
	\end{enumerate}
\end{theorem}
\begin{proof}
By decomposing 
	\begin{align*}
		Z^\Phi_n= Z^{\Phi-\E\Phi}_n + Z^{\E\Phi}_n
	\end{align*}
	it is enough to  get the desired bound on each term separately.
	If  $\mathrm{v}$ is the right eigenvector of $A$ to eigenvalue $\rho>1$ then 
	\begin{align*}
		\big<{\bf 1}, Z_n\big>\le \min_i \mathrm{v}_i^{-1}\big<\mathrm{v}, Z_n\big>
		=\rho^{n}\min_i \mathrm{v}_i^{-1}\big\langle\mathrm{v},W_n^{(1)}\big\rangle,
	\end{align*}
	and further
	\begin{align*}
		\big|Z_n^{\E\Phi}\big|& \le  \sum_{k\ge 0} \|\E\Phi(n-k)\| \|Z_k\|
		\le  \sum_{k\ge 0} \|\E\Phi(n-k)\| \big<{\bf 1}, Z_k\big>\\
		&\le \min_i\mathrm{v}_i^{-1}\sum_{k\ge 0} \|\E\Phi(n-k)\|\big<\mathrm{v}, Z_k\big>
		\le \min_i\mathrm{v}_i^{-1}\sum_{k\ge 0} \|\E\Phi(n-k)\|\big<\mathrm{v}, \rho^k W_k^{(1)}\big>\\
		&=\rho^n \times \min_i\mathrm{v}_i^{-1}\sum_{k\le n} \|\E\Phi(k)\| \rho^{-k}\big<\mathrm{v}, W_{n-k}^{(1)}\big>.
	\end{align*}
	In view of Lemma 2.2 in  \cite{kolesko-sava-huss-char} the random variables $\big<\mathrm{v}, W_k^{(1)}\big>$ are bounded in $\L^2$ and therefore by Minkowski inequality the $\L^2$ norm of $Z_n^{\E\Phi}$ is bounded by a multiple of $\rho^n\sum_{k\le n} \|\E\Phi(k)\| \rho^{-k}$. As a result we get 
		\begin{align*}
		\E\Big[\big|Z_n^{\E\Phi}\big|^2\Big]=O(\rho^{2n}),
	\end{align*}
	in case \textit{ i)} and 
		\begin{align*}
		\E\Big[\big|Z_n^{\E\Phi}\big|^2\Big]=O(n^{2r}\rho^{2n}),
	\end{align*} in case  \textit{ii)}. By Jensen inequality  the $p$th moment of $Z_n^{\E\Phi}$, for $p\in[1,2]$ is of order $O(\rho^{pn})$  in case \textit{ i)} and of order $O(n^{pr}\rho^{pn})$ in case \textit{ ii)} respectively.

Now we focus on the case with a centered characteristic $\Psi:=\Phi-\E\Phi$.	
	We consider an increasing sequence $(G_n)_{n\in\N}$ of subsets of $\U$ with the following property:
		$\cup_{n\geq 1}G_n=\U$; for any $n\in\N$, $G_n=n$; if $u\in G_n$, then for any $v\leq u$, $v\in G_n$.
		Such a sequence can be constructed by using the diagonal method. If $\mathcal{G}_n=\sigma\left(\{L(u):\ u\in G_n\}\right)$, then one can see that  $\sum_{u\in G_k}\Psi_u(n-|u|)\mathrm{e}_{\type{(u)}}$ is a $\mathcal{G}_k$-martingale. Indeed, for any $u\in G_k$ both $\type(u)$ and $\Psi_u$ are $\mathcal{G}_k$-measurable and the fact that $\Psi$ is centered gives the martingale property.
	By the Topchii-Vatutin inequality \cite[Theorem 2]{Topchii+Vatutin:1997} applied to $\sum_{u\in G_n}\Psi_u(n-|u|)\mathrm{e}_{\type{(u)}}$ we get
	\begin{align*}
	\E\Big[\Big|\sum_{u\in\T}&\Psi_u(n-|u|)\mathrm{e}_{\type(u)}\Big|^p\Big]\le C_p\E\Big[\sum_{u\in\T}\big|\Psi_u(n-|u|)\mathrm{e}_{\type(u)}\big|^p\Big]\\
	&\le C_p \sum_{k\ge 0}\E\big[\|\Psi(n-k)\|^p\big]  \rho^{k} 
	\le 2^p C_p \sum_{k\ge 0}  \E\big[\|\Phi(n-k)\|^p \big]\rho^k\\
	&\le 2^p C_p \Big(\sum_{k\ge 0}  \big(\E\big[\|\Phi(n-k)\|^p \big]\big)^{1/p}\rho^{k/p}\Big)^p
	\le 2^p C_p \Big(\sum_{k\ge 0}  \big(\E\big[\|\Phi(n-k)\|^p \big]\big)^{1/p}\rho^k\Big)^p\\
	&\le  2^p C_p\rho^{pn} \Big(\sum_{k\le n}  \big(\E\big[\|\Phi(k)\|^p \big]\big)^{1/p}\rho^{-k}\Big)^p
	\end{align*}
	and the latter expression is of the order $O(\rho^{pn})$ in case \textit{i)} or $O(n^{pr}\rho^{pn})$ in case \textit{ii)}. This finishes the proof of the first two statements  \textit{i)} and \textit{ii)}.
	
Now we turn to the proof of \textit{iii)} and \textit{iv)}.   Observe that Burkholder-Davis-Gundy inequality \cite[Theorem 1.1]{burkholder1972integral} yields
	\begin{align*}
		\E\Big[\big|Z_n^{\Phi}\big|^{2p}\Big]&=
		\E\Big[\Big|\sum_u \Phi_u(n-|u|)\mathrm{e}_{\type(u)}\Big|^{2p}\Big]\\
		&\le C_p\E\Big[\Big|\sum_u \big|\Phi_u(n-|u|)\mathrm{e}_{\type(u)}\big|^2\Big|^{p}\Big]
		=C_p\E\Big[\big|Z_n^{\Psi}\big|^{p}\Big],
	\end{align*}
	where $\Psi$ is a new characteristic defined by $\Psi_u(k)\mathrm{e}_i\defeq\big|\Phi_u(k)\mathrm{e}_{i}\big|^2$, i.e.~the components of $\Psi(k)$ are square of components of $\Phi$. Clearly, $\|\Psi(k)\|\le \|\Phi(k)\|^2$ and as a consequence  \textit{iii)} and \textit{iv)} follow from parts \textit{i)} and \textit{ii)} respectively applied to the characteristic $\Psi$.

\end{proof}
\begin{corollary}
	\label{cor:moment estimate}
Let $\Phi:\Z\to \C^{1\times J}$ be a random characteristic such that 
\begin{equation}
	\sum_{k\in \Z}\big\|\E[\Phi(k)]\big\|(\rho^{-k}+\vartheta^{-k})<\infty,
\end{equation}
for some $\vartheta<\sqrt{\rho}$, and
\begin{equation}
	\sum_{k\in \Z}\|\Var[\Phi(k)]\|\rho^{-k}<\infty.
\end{equation}
Suppose that $\E\left[\|L\|^{2p}\right]<\infty$ for some $p\in(1,2)$.
Then, for $F_n^\Phi$ defined by \eqref{eq:fluctuations-phi} it holds
	$$\E\Big[\big|Z_n^{\Phi}-F^\Phi_n\big|^{2p}\Big]=
	\begin{cases}
		O\big(\rho^{np}\big)&\text{if for all $0\leq \ell \leq J-1$, }\sigma_\ell=0,\\
		O\big(n^{(2\ell+1)p}\rho^{np}\big)\quad&\text{if $0\leq \ell\leq J-1$ is maximal with }\sigma_\ell>0.
	\end{cases}$$
\end{corollary}
\begin{proof}
The decomposition from equation (18) in \cite{kolesko-sava-huss-char} yields
	\begin{align}\label{eq:decomposition}
		Z_n^{\Phi}-F^\Phi_n=Z_n^{\Psi_1}+Z_n^{\Psi_2}+o(\rho^{n/2}),
	\end{align}
	where $\Psi_1$ and $\Psi_2$ are two random centered characteristics such that
	\vspace{-0.3cm}
	\begin{itemize}\setlength\itemsep{-0.2em}
		\item for any  $k\in \Z$ we can decompose $\Psi_1$ as $\Psi_1(k)=\Psi_1'(k)(L-A)$ for some deterministic  characteristic $\Psi_1'(k)$ (see the paragraph after equation (19) in \cite{kolesko-sava-huss-char}) and  
		$\sum_{k\in \Z}\E\left[\|\Psi_1(k)\|^2\right]\rho^{-k}<\infty$,
		\item $Z_n^{\Psi_2}=\mathrm{x}_2\pi^{(2)} (Z_n-A^n_2Z_0)$, that is  $\Psi_2(k)=\mathrm{x}_2\pi^{(2)}A^{k-1}(L-A)\ind{k>0}$, for some row vector $\mathrm{x}_2$.
	\end{itemize}
	Moreover, the last term $o(\rho^{n/2})$ in \eqref{eq:decomposition} is deterministic.
By Minkowski's inequality, we have
\begin{equation}\label{eq:mink}
\E\Big[\big|Z_n^{\Phi}-F^\Phi_n\big|^{2p}\Big]^{1/2p}\leq \E\Big[\big(Z_n^{\Psi_1}\big)^{2p}\Big]^{1/2p}+ \E\Big[\big(Z_n^{\Psi_2}\big)^{2p}\Big]^{1/2p}+o(\rho^{n/2}),
\end{equation}
and we estimate each of the two terms on the right hand side separately.	
In view of Lemma \ref{lem:hypercontractivity}, there is a constant $C>0$ such that 
	$$\E\left[\|\Psi_1(k)\|^{2p}\right]\le C\Big(\E\left[\|\Psi_1(k)\|^2\right]\Big)^{p},$$
	and, in particular,
	\begin{align*}
	\sum_{k \in\Z}\rho^{-k}\big(\E\left[\|\Psi_1(k)\|^{2p}\right]\big)^{1/p}
	\le C^{1/p}\sum_{k\in\Z}\rho^{-k}\E\left[\|\Psi_1(k)\|^{2}\right] <\infty.
%	\kk{\le C\sum_{k\le n}\Big(\sum_{k\in\Z}\rho^{-k}\E\left[\|\Psi_1(k)\|^{2}\right]\Big)^p\rho^{k(p-1)}}
%	\kk{=\frac{C\rho^{n(p-1)}}{\rho^{p-1}}}  \Big(\sum_{k\in \Z}\rho^{-k}\E\left[\|\Psi_1(k)\|^{2}\right]\Big)^p.
	\end{align*}
Theorem \ref{lem:moments growth} {\it iii)} applied to $\Psi_1$ yields $\E\big[\big(Z^{\Psi_1}_{n}\big)^{2p}\big]=O(\rho^{np})$.
	In order to deal with the second term $\E\left[\big(Z^{\Psi_2}_{n}\big)^{2p}\right]$ on the left hand side of \eqref{eq:mink}, note that by the definition of $\ell$ we may write
	\begin{align*}
	\Psi_2(k)&=\ind{k>0}\mathrm{x}_2\pi^{(2)}  \sum_{j=0}^{J-1}{k-1 \choose j}D^{k-1-j}N^{j}(L-A)\\
	&=\ind{k>0}\mathrm{x}_2\pi^{(2)}  \sum_{j=0}^{(J-1)\wedge \ell}{k-1 \choose j}D^{k-1-j}N^{j}(L-A).
	\end{align*}
 In particular, as $k$ goes to infinity we have
	$$\big(\E\left[\|\Psi_2(k)\|^{2p}\right]\big)^{1/p}=O\big(\rho^k k^{2\ell}\big),\quad \text{so}\quad 
	\sum_{k=0}^n\rho^{-k}\big(\E\left[\|\Psi_2(k)\|^{2p}\right]\big)^{1/p}=O\big(n^{2\ell+1}\big),$$
	and by 	Theorem \ref{lem:moments growth} {\it iv)} applied to $\Psi_2$, we obtain $\E\big[\big(Z^{\Psi_2}_{n}\big)^{2p}\big]=O(n^{(2\ell+1)p}\rho^{np})$, which together with \eqref{eq:mink} proves the desired.
%	\textcolor{blue}{Why do we write here (2l+1) and not $2l$?}
\end{proof}

For any function $f:\R\to\R$ let $$\omega(f,k,t)\defeq\sup\{|f(x)-f(y)|:x,y\in[-k,k],\ |x-y|\le t\}$$
be the modulus of continuity of $f$ on the interval $[-k,k]$.
\begin{lemma}
	\label{lem:auxiliary lemma for cadlag processes}
	For  a stochastic process $X$ taking values in the Skorokhod space  $\mathcal D(\R)$, and for  $n\in \N$ let $X_n(t)= X(t+n)$. Further, suppose that $Y=(Y(t))_{t\in\R}$ is a stationary process with almost surely continuous trajectories. Then we have the following.
	\vspace{-0.3cm}
	\begin{enumerate}[i)]\setlength\itemsep{-0.2em}
		\item If $X_n\distto Y$, $N_n$ is a sequence of natural numbers diverging to infinity, and $\delta_n\searrow0$, then there is a sequence $k_n\nearrow\infty$ such that 
		\vspace{-0.25cm}
		\begin{align}
		\label{eq:vanishing modulus of continuity}
		\omega(X_{N_{n}},k_n,\delta_{n})\Probto0,\quad\text{as }n\to\infty.
		\end{align}
		\item If for some real valued random variable $S$  independent of $Y$, it holds
		$(S,X_n)\distto(S,Y)$ as $n\to\infty$, then for any sequence $a_n$ that diverges to infinity we have
		\begin{align}
		\label{eq:converges at random time}
		X(a_n+S)\distto Y(0),\quad \text{as } n\to\infty .
		\end{align}
	\end{enumerate}
\end{lemma}
\begin{proof} \textit{i)}:
	Fix $\delta>0$ and $k\in \N$. Then the mapping $\mathcal D(\R)\ni f\mapsto \omega(f,k,\delta)\in\R$ is continuous at any  $f\in\mathcal C(\R)$.
\begin{comment}	
	 This can be immediately seen by
fixing $k\in\N$ and $\delta>0$, and letting $f,g\in \mathcal{C}(\R)$ be two continuous functions over $\R$. Then for any $x,y\in[-k,k]$ with $|x-y|\leq \delta$ we have
$$|f(x)-f(y)|\leq |f(x)-g(x)|+|g(x)-g(y)|+|g(y)-f(y)|\leq 2||f-g||_{\infty}+\omega (g,k,\delta)$$
since on the compact interval $[-k,k]$, $||f-g||_{\infty}$ is well-defined, which implies that
$$\omega(f,k,\delta)\leq \omega (g,k,\delta)+2||f-g||_{\infty}.$$
So, from the previous two equations we deduce that for any fixed $k\in \N$ and $\delta>0$
$$|\omega(f,k,\delta)-\omega (g,k,\delta)|\leq 2 ||f-g||_{\infty} $$
which proves that the mapping  $\mathcal C(\R)\ni f\mapsto \omega(f,k,\delta)\in\R$ is Lipschitz continuous, and in particular, continuous.
\end{comment}
The continuous mapping theorem yields
	$$\lim_{n\to\infty} \E[\omega(X_{N_n},k,\delta)\wedge 1]=\E[\omega(Y,k,\delta)\wedge 1].$$ 
	Hence, for any $\varepsilon >0$ we have
	$$\limsup_{n\to\infty} \E[\omega(X_{N_n},k,\delta_{n})\wedge 1]\le\E[\omega(Y,k,\varepsilon)\wedge 1].$$
	Since $\varepsilon$ is arbitrary, we can take the limit as $\varepsilon$ goes to 0, and from the continuity of $Y$ we conclude that
	$$\E[\omega(X_{N_n},k,\delta_{n})\wedge 1]\to0.$$ In particular, there is $n(k)>n(k-1)$  such that for all $n\ge n(k)$ it holds 
	$$\E[\omega(X_{N_{n}},k,\delta_n)\wedge 1]\le1/k.$$
	
	Now for $n\ge n(1)$ we set $k_n=k$ whenever $n(k)\le n<n(k+1)$. Clearly $k_n\nearrow\infty$ and  it holds
	$$\E[\omega(X_{N_{n}},k_n,\delta_n)\wedge 1]\le1/k_n,$$	
	which implies \eqref{eq:vanishing modulus of continuity} and this proves the first part of the claim.
	
\textit{ii)}: The convergence in \eqref{eq:converges at random time} holds if for any subsequence $n_k$ we can choose a further subsequence $n_{k_l}$ along which the convergence holds. Since we may replace the sequence $a_n$ by a subsequence $a_{n_k}$ it suffices if we show the convergence \eqref{eq:converges at random time} along some subsequence. Thus, without loss of generality, we may assume that $\{a_n\}\to a$ for some $a\in[0,1]$.
	For $N_n=\lfloor a_n\rfloor$, $\delta_n=2|\{a_n\}-a|$ from part \textit{i)} we infer the existence of a sequence $k_n\nearrow\infty$ such that \eqref{eq:vanishing modulus of continuity} holds.
	
	For $Z_n= X(N_n+\{a_n\}+S)-X(N_n+a+S)$, once again  by  part\textit{ i)} of the proof, we have
\begin{align*}
|Z_n|\le|Z_n|\ind{|S|\le k_n}+|Z_n|\ind{|S|> k_n}
\le \omega(X_{N_{n}},k_n,\delta_{n})+|Z_n|\ind{|S|> k_n}\Probto0,
\end{align*}
and thus, by Slutsky's theorem, it suffices to prove 
\begin{align*}
X_{N_n}(a+S)=X(N_n+a+S)\distto Y(0).
\end{align*}
Next, observe that the mapping
$\R\times\mathcal D(\R)\ni(s,x)\mapsto x(a+s)\in \R$
	is continuous at any point $(s,x)\in \R\times \mathcal C(\R)$. 
	%This can be proven as in the first part, by showing that the mapping is Lipschitz continuous.
% \kk{Do we need to prove it?}.	
%	\textcolor{blue}{I don't think that we have to prove it: it is an $\epsilon-\delta$ argument. The open sets in $\mathbb{R}\times \mathcal{C}(\mathbb{R})$ are of the form: for $(x,f)\in \mathbb{R}\times \mathcal{C}(\mathbb{R}) $, and $\delta,\epsilon>0$, a $(\delta,\epsilon)$-ball around $(x,f)$ is of the form
%\begin{align*}
%&B_{(\delta,\epsilon)}(x,f)=\\
%&\left\{(y,g)\in \mathbb{R}\times \mathcal{C}(\mathbb{R}):\ |x-y|<\delta; \ |f(t)-g(t)|<\epsilon/2,
%\text{ for all } t\in [-h,h]
%\text{ with } [x-a,x+a]\subseteq [-h,h]\right\}
%\end{align*}
%where $h$ is the smallest number such that the interval $[-h,h]$ contains the interval $[x-a,x+a]$. On the compact interval $[-h,h]$, $\sup_{t\in [-h,h]}|f(t)-g(t)|$ is then well defined.
%	 Now we have to show that if $(y,g)\in B_{(\delta,\epsilon)}(x,f)$ for $\epsilon,\delta\to 0$, that is if $(y,g)$ is sufficiently close to $(x,f)$ than $|f(x+a)-g(x+a)|$ goes to $0$.
%	\begin{align*}
%	|f(x+a)-g(y+a)|\leq |f(x+a)-f(y+a)|+ |f(y+a)-g(y+a)|\leq \epsilon/2+\epsilon/2	=\epsilon
%	\end{align*}
%	where $|f(x+a)-f(y+a)|\leq \epsilon/2$ because f is continuous and $|x-y|<\delta$. For the second $\epsilon/2$, since	
%	$y\in [x-\delta,x+\delta]\subset [x-a,x+a]$ we have $y+a \in [x,x+2a]\subset [-(h+1),h+1]$  and we can use the definition of the open sets above. Maybe to replace $h$ by $h+1$.}
Therefore, by the continuous mapping theorem we have 
	$$X_{N_n}(a+S)\distto Y(a+S)\eqdist Y(0),$$
	and this completes the proof.
\end{proof}
\begin{lemma}
	\label{lem:hypercontractivity}
	Suppose that $X$ is a random $k\times m$ matrix with $\E[\|X\|^r]<\infty$ for some $r>1$. Then for any $q<r$, there is a constant $C=C(m,k,q,r,X)>0$ such that for any $m\times k$ deterministic matrix $A$ it holds
	$$\E[\|AX\|^r]\le C\E[\|AX\|^q]^{r/q}.$$
\end{lemma}
\begin{proof}
%	First let us note that is suffices to show that it holds for  $1\times n$ matrices, as the left hand above can be bounded by finite sum of such terms. Hence our problem boils down to to get the following estimate:
%	$$\E\big[\big|\langleA,X\rangle\big|^r\big]\le C\E\big[\big|\langleA,X\rangle\big|^q\big]^{r/q},$$
%	for any vector $A\in\R^n$.
	Without loss of generality, by the homogeneity of both sides, we may also assume that $\|A\|= 1$.
	Let now $N=\{a\in\R^{m\times k}:aX=0\text{ a.s.}\}$ be a linear subspace of $\R^{m\times k}$ and $V$ its orthogonal complement. As both functions
	$$a\mapsto \E[\|aX\|^r]\quad \text{and}\quad a\mapsto\E[\|aX\|^q]^{r/q}$$
	defined on the compact space $V\cap\{\|x\|=1\}$ are continuous and do not vanish, they achieve their minimum and maximum. We define
	$$C=\frac{\max_{a\in V,\|a\|=1} \E[\|aX\|^r] }{\min_{a\in V,\|a\|=1}\E[\|aX\|^q]^{r/q}}<\infty.$$
	Finally, by splitting $A=A_1+A_2$ with $A_1\in N$ and $A_2\in V$ we then have 
	$$\E[\|AX\|^r]\le\E[\|A_2X\|^r]\le C\E[\|A_2X\|^q]^{r/q}\le C\E[\|AX\|^q]^{r/q},$$
	and this completes the proof.
\end{proof}
\begin{lemma}
	\label{lem:auxilarity lemma}
	Let $I$, $J$ be  two disjoint subintervals of $[0,1]$, $N\in\N$ and let $U_1,\dots U_N$ be an independent collection of random variables  uniformly distributed on $[0,1]$. Then for any sequence $a\in\ell^2$ and any numbers $A,B\in \R$, we have 
	\begin{align*}
		\E\Big[\Big|\sum_{i=1}^N(\ind{U_i\in I}-|I|)a_i& +A\Big|^2\cdot\Big|\sum_{i=1}^N(\ind{U_i\in J}-|J|)a_i+B\Big|^2\Big] \\
		& \leq C|I| |J|\|a\|^2(A^2+B^2+\|a\|^2)+A^2B^2\\
		& \lesssim |I||J|\|a\|^4 +A^4+B^4,
	\end{align*}
for some absolute constant $C>0$.
\end{lemma}
\begin{proof}
	For ease of  notation, for $i=1,\ldots,N$ we set 
	$$q_i= \big(\ind{U_i\in I}-|I|\big)a_i\quad \text{ and } r_i=\big(\ind{U_i\in J}-|J|\big)a_i,$$
	so $\E[q_i]=\E[r_i]=\E[q_i q_j]=\E[r_i r_j]=0$, for $i\ne j$. Simple calculations give
\begin{align*}
	\E[q_i^2]&=\big(|I|-|I|^2\big)a_i^2,\quad \text{and}\quad 	\E[q_ir_i]=-|I||J|a_i^2, \\
	\E[q_i^2r_i]&=\big(-|I||J|+2|I|^2|J|\big)a_i^3,\quad \text{and} \quad 
	\E[q_i^2r_i^2]=\big(|I||J|(|I|-3|I||J|+|J|)\big)a_i^4.
	\end{align*}
	We first have
	\begin{equation}\label{eq:expec-2}
	E\Big[\big(\sum_{i\leq N}q_i+A\big)^2\Big]=(|I|-|I|^2)N+A^2\leq |I|N+A^2.
	\end{equation}
	Expanding the expectation we get
	\begin{align*}
	\E\Big[\Big|\sum_{i\le N}q_i+A\Big|^2\cdot\Big|\sum_{i\le N}r_i+B\Big|^2\Big]=
	\sum_{i_1,i_2,j_1,j_2}\E[q_{i_1}q_{i_2}r_{j_1}r_{j_2}]+2A\sum_{i,j_1,j_2}\E[q_ir_{j_1}r_{j_2}]\\
	+2B\sum_{i_1,i_2,j}\E[q_{i_1}q_{i_2}r_{j}]+4AB\sum_{i,j}\E[q_ir_j]+A^2B^2\eqdef I+II+III+IV+A^2B^2.
	\end{align*}
	We show that each of the four terms $I,II, III, IV$ is bounded by a multiple of $|I| |J|N(A^2+B^2+N)$.
	Note that a nontrivial term of the form $\E[q_{i_1}q_{i_2}r_{j_1}r_{j_2}]$ is either of the form $E[q_i^2r_j^2]$ or $\E[q_{i_1}q_{i_2}r_{i_1}r_{i_2}]$, which in turn implies
	\begin{align*}
	I&=\sum_{i,j}\E[q_i^2r_j^2]+4\sum_{i\neq j}\E[q_ir_i]\E[q_jr_j]=|I||J|(|I|-3|I||J|+|J|)\sum a_i^4\\
	& +2(|I|-|I|^2)(|J|-|J|^2)\sum_{i\neq j}a_i^2a_j^2+4|I|^2|J|^2\sum_{i\neq j}a_i^2a_j^2
	 \le 8|I||J|\|a\|^4.
	\end{align*}
	Next, $\E[q_{i_1}q_{i_2}r_{j}]$ is nonzero if it is of the form $\E[q_i^2r_i]$. Hence, we have
	\begin{align*}
	III  = 2B\sum_i\E[q^2_ir_i]=2B(-|I||J|+2|I|^2|J|)\sum_{i}a_i^3
	 =2B|I||J|(2|I|-1)\sum_{i}a_i^3 	\le 4|I||J||B|\|a\|^3.
	\end{align*}
	By symmetry, 	we have
	\begin{align*}
	II=2A(-|I||J|+2|I||J|^2)\sum_{i}a_i^3=2A|I||J|(2|J|-1)\sum_{i}a_i^3
	\le 4|I||J||A|\|a\|^3.
\end{align*}
Finally, by the same reasoning as above, we get
\begin{align*}
IV=-4AB|I||J|\sum_{i}a_i^2\le 4|AB||I||J|\|a\|^2,
\end{align*}
and the claim follows by putting together the four quantities.
\end{proof}

\begin{lemma}
	\label{lem:linear independence}
	Let $l,N\in\N$ and let  $\lambda_1,\dots,\lambda_\ell$ be different, non-zero complex numbers. For $(i,j)\in\N_0\times [\ell]$ we define  $f_{i,j}:\Z\to\C$  by $f_{i,j}(k)= k^i\lambda_j^k$. Then the collection of functions $\{f_{i,j}:(i,j)\in\N_0\times [\ell]\}$ is linearly independent. In particular, if for any $i\in\N_0$, $j\le \ell$, $p_{j,i}$ is a polynomial of degree $i$, then the collection $\{k\mapsto \lambda_j^{k} p_{j,i}(k):(i,j)\in\N_0\times [\ell]\}$ is also linearly independent.
\end{lemma}
\begin{proof}
	Let $h=\sum_{i,j}c_{i,j}f_{i,j}$ be a finite linear combination of the functions $f_{i,j}$. Our aim is to show that if
	\begin{align}
	\label{eq:linear independence}
	h(k)=0\text{ for }k\ge N\text{ then }c_{i,j}=0 \text{ for all }i,j.
	\end{align}
	By $d_j(h)$ we denote the maximal power $i$ such that the element $f_{i-1,j}$ appears in the combination of $h$. Furthermore, we set $d(h)= d_1(h)+\dots d_\ell(h)$. We prove the claim by induction on $d(h)$.
		If $d(h)=1$ then $h(k)=\lambda_j^k$ for some $j$ and the conclusion follows. Now suppose that \eqref{eq:linear independence} holds for any $h$ with $d(h)=n$ and take $h$ with $d(h)=n+1$. If $d_j(h)\le1$ for all $j\le \ell$ then $h(k)=\sum_{m=1}^{n+1}c_{j_m}\lambda_{j_m}^k$ for some $j_1,\dots,j_{n+1}\le n+1$. Since $\big(\lambda_{j_m}^{-N+1}f_{0,j_m}(k)\big)_{1\le k,m\le n+1}$ forms a Vandermonde matrix, this implies \eqref{eq:linear independence}.
	
	Therefore, we may now assume that for some $j_0\le \ell$ we have $d_{j_0}(h)\ge2$. We denote by $\nabla$ the difference operator defined by $\nabla f(k)= f(k+1)-f(k)$, and  $m_{j_0}$ is defined by $m_{j_0}f(k)=\lambda_{j_0}^kf(k)$. We now define a linear operator $\nabla_{{j_0}}= m_{j_0}\nabla m_{j_0}^{-1}$. Clearly $\nabla_{j_0}$ acts on the linear combinations $g$ of  $f_{i,j}$ with 
	$d_{j}(\nabla_{j_0}g)\le d_{j}(g)$ and $d_{j_0}(\nabla_{j_0}f_{i,j_0})=d_{j_0}(f_{i,j_0})-1 $. In particular, $1\le \nabla_{j_0}h\le n$, and hence by the induction hypothesis  $\nabla_{j_0}h(k)\neq0$ for some $k\ge N$, which finally implies that $h(k)\neq0$ or $h(k+1)\neq0$, thus proving \eqref{eq:linear independence}.
\end{proof}

\paragraph{Acknowledgments.}The research of Ecaterina Sava-Huss is supported by the Austrian Science Fund (FWF): P 34129.
We are very grateful to the two anonymous referees for their suggestions and positive criticism that improved substantially the quality and the presentation of the paper.
\vspace{-0.5cm}
%\bibliography{range}{}
\bibliography{GW}{}
\bibliographystyle{abbrv}

\textsc{Konrad Kolesko}, Department of Mathematics, University of Wroclaw, Poland.\\
\texttt{Konrad.Kolesko@math.uni.wroc.pl}

\textsc{Ecaterina Sava-Huss}, Institut für Mathematik, Universität Innsbruck, Österreich.\\
\texttt{Ecaterina.Sava-Huss@uibk.ac.at}
\end{document}